\documentclass[11pt,reqno]{amsart}
\usepackage{amsfonts}
\usepackage{amscd}
\usepackage{amssymb}
\usepackage{amsfonts}
\usepackage{amscd}
\usepackage{amsmath} 
\usepackage{mathrsfs}
\usepackage{dsfont}
\usepackage{enumerate}
\usepackage{xcolor}
\usepackage{pxfonts}
\usepackage{float} 
\usepackage{graphicx}
\allowdisplaybreaks
\usepackage{url}
\usepackage[hidelinks]{hyperref} 
 \hypersetup{
     colorlinks,
     linkcolor={red!70!black},
     citecolor={blue!80!black},
     urlcolor={blue!60!black}
 }

\date{\today}

\newtheorem*{theorem*}{Theorem}
\newtheorem{theorem}{Theorem}[section]

\newtheorem{lemma}[theorem]{Lemma}
\newtheorem{proposition}[theorem]{Proposition}
\theoremstyle{definition}
\newtheorem{definition}[theorem]{Definition}
\theoremstyle{remark}
\newtheorem{remark}[theorem]{\bf{Remark}}

\newtheorem*{assumption}{\textbf{Assumption}}

\numberwithin{equation}{section}

\newcommand{\R}{\mathbb R}
\newcommand{\hn}{\mathbb{H}^n}

\title[Inhomogeneous CR Yamabe Problem on the Heisenberg Group]{Profile decomposition and multiple positive solutions for the perturbed CR Yamabe equation on the Heisenberg group}
    \author[Basak]{Riju Basak}
	\address[R. Basak]{Department of Mathematics, National Taiwan Normal University, No. 88, Section 4, Tingzhou Road, Wenshan District, Taipei City, Taiwan 116, R.O.C.
	}
	\email{rijubasak52@gmail.com}

    \author[Chakraborty]{Souptik Chakraborty}
    \address[S. Chakraborty]{Gandhi Institute of Technology and Management (GITAM) University, Department of Mathematics and Statistics, Hyderabad, Telangana,  502329, India.
    }
	\email{soupchak9492@gmail.com, schakrab8@gitam.edu}
    
	\author[Rana]{Tapendu Rana}
	\address[T. Rana]{Department of Mathematics: Analysis, Logic and Discrete Mathematics, Ghent University, Krijgslaan 281, Building S8, B 9000 Ghent, Belgium.
	}
	\email{tapendurana@gmail.com, tapendu.rana@ugent.be}
    
	\author[Roychowdhury]{Prasun Roychowdhury}
    \address[P. Roychowdhury]{Department of Mathematics, Indian Institute of Technology Hyderabad, Kandi, Sangareddy, Telangana, 502285, India.
    } 
	\email{prasunrc@math.iith.ac.in}

\date{}
   
\begin{document}
\begin{abstract}
In this article, we study the inhomogeneous critical equation
\begin{align*}
  \mathcal{L}_{\mathbb H^n} u=a(\xi)|u|^{2^\star-2}u+f
\quad \text{in } \mathbb H^n,
\qquad u\in S^{1,2}(\mathbb H^n),
\end{align*}
where $\mathcal{L}_{\mathbb H^n}$ is the sub-Laplacian on the Heisenberg group, $2^\star=\frac{2Q}{Q-2}$ with $Q=2n+2$, $a$ is a positive continuous coefficient tending to $1$ at infinity, and $f$ is a nontrivial nonnegative functional in $(S^{1,2}(\mathbb H^n))'$. When $a\equiv 1$, this equation can be seen as a perturbation of the CR Yamabe equation. We first establish a Palais--Smale profile decomposition for the associated energy functional on the whole Heisenberg group. We then show that this decomposition is sharp, in the sense that nonnegative Palais--Smale sequences lose compactness exactly at the energy levels predicted by it, and these levels are explicit in terms of the best Folland--Stein constant and the values of $a$. As an application, we prove that the equation admits at least two positive solutions whenever $f$ is suitably small and $a$ is sufficiently flat near one of its maximum points. As a key analytic ingredient, we establish an improved Folland-Stein-Sobolev inequality involving the Morrey norm, which serves as a fundamental interpolation inequality and plays a crucial role in detecting the concentration of noncompact Palais-Smale sequences.
\end{abstract}
\maketitle

\section{Introduction}
The Yamabe problem, the question of finding a metric of constant scalar curvature within a given conformal class, occupies a central place in geometric analysis, sitting at the intersection of differential geometry, partial differential equations, and the calculus of variations. Analytically, it reduces to solving a nonlinear elliptic partial differential equation involving the critical Sobolev exponent \cite{Sch84, LP87}. This criticality causes the Sobolev embedding to fail to be compact, so that minimizing (or Palais--Smale) sequences for the associated variational functional need not converge, but may instead concentrate and form \emph{bubbles} modeled on the Aubin-Talenti extremals \cite{Aub76, Tal76}. Understanding this loss of compactness has been a central theme since the work of Brézis and Nirenberg \cite{BN83}, developed further by P.~L.~Lions \cite{Lio84a, Lio84b}, Struwe \cite{Str84}, Bahri and Coron \cite{BC88}, and Tarantello \cite{Tar92}, among many others. In the sub-Riemannian setting, the CR Yamabe problem on strictly pseudoconvex CR manifolds leads to a critical equation modeled on the Heisenberg group, where the loss of compactness is described by bubbles adapted to its non-isotropic structure.

A basic tool for analysing this lack of compactness is the \emph{Palais--Smale profile decomposition}. In his seminal work, Struwe \cite{Str84} demonstrated that along any Palais--Smale sequence for critical elliptic problems on bounded Euclidean domains, compactness is lost exclusively through the formation of finitely many concentrating bubbles modeled on the Aubin--Talenti solutions \cite{Aub76, Tal76}. This characterization was subsequently extended to the full space $\dot{H}^1(\mathbb{R}^N)$ by Solimini \cite{Sol95}, reformulated via wavelet decompositions by Gérard \cite{Ger98} and Jaffard \cite{Jaf99}, and developed for fractional Sobolev spaces by Palatucci and Pisante \cite{PP14, PP15}. In the \emph{nonhomogeneous} setting, Tarantello \cite{Tar92} established a pioneering multiplicity result on bounded Euclidean domains, proving the existence of at least two distinct solutions when the nonhomogeneous perturbation is suitably small. More recently, this framework was extended to the full domain $\R^N$  by Bhakta and Pucci \cite{BP20}, who established a global profile decomposition and obtained multiple positive solutions for the fractional Laplacian on $\mathbb{R}^N$. Related multiplicity and compactness phenomena for nonlocal scalar field equations were further explored in \cite{BCG23}.
\subsection*{The subelliptic setting}
The Heisenberg group $\mathbb{H}^n$ is the model space in sub-Riemannian and CR geometry. Let $Q=2n+2$ be its homogeneous dimension and $2^{\star}=\frac{2Q}{Q-2}$ the critical Folland--Stein exponent. The Folland--Stein embedding \cite{FS74,FS82} states that the homogeneous Sobolev space $S^{1,2}(\mathbb{H}^n)$ embeds continuously, but not compactly, into $L^{2^{\star}}(\mathbb{H}^n)$. For a detailed description of the space $S^{1,2}(\hn)$, see \eqref{hom-sobl}, and for notational convenience, throughout the remainder of the paper we denote $S^{1,2}(\hn)$ by $S^1(\hn)$. The failure of compactness is driven by the invariance of the space under left-translations and non-isotropic dilations $\delta_\lambda$, and the corresponding critical profiles are the \emph{Jerison--Lee bubbles}. The subelliptic CR Yamabe equation is given by the homogeneous critical problem
\begin{equation}\label{hom_eqn}
    \mathcal{L}_{\mathbb{H}^n}u = |u|^{2^{\star}-2}u,
    \quad u>0,\quad u\in S^1(\mathbb{H}^n),
\end{equation}
where $\mathcal{L}_{\mathbb{H}^n}$ denotes the sub-Laplacian on $\mathbb{H}^n$ defined in \eqref{eq_def_sub_lap}. In their seminal work, Jerison and Lee \cite[Corollary~C]{JL88} completely classified all finite energy positive solutions to \eqref{hom_eqn}. In fact, subelliptic regularity theory implies that any positive weak solution is classical; see, for instance, \cite{GV00} and the discussion in \cite[Section~3]{CU01}. We also mention that further classification results, under decay conditions at infinity or even without any integrability assumption, were obtained in \cite{CLMR25, FV23, PV25, Liu25}. Since all the solutions considered in this article belong to $S^1(\mathbb{H}^n)$, we shall not need these extensions. For the reader's convenience, we record the classification in the form we use.

\begin{theorem}\label{thm_JL_bubble}
Suppose $u\in S^{1}(\mathbb{H}^n)$ is a positive weak solution of \eqref{hom_eqn}. Then $u$ is a smooth function and takes the form of a left-translation and dilation of the standard Jerison-Lee bubble, i.e., 
\begin{align}\label{eq_jer_Lee_bub}
u(\xi) =D_{\eta,\lambda}U(\xi)\coloneqq \lambda^{-\frac{Q-2}{2}}U\left(\delta_{\frac{1}{\lambda}}(\eta^{-1}\circ\xi)\right),    
\end{align}
for some $\lambda>0$ and $\eta\in\mathbb{H}^n$, where for $\xi=(z,t) \in  \mathbb{H}^n$,
\begin{align}\label{eq_J_L_extreme}
U(\xi) \coloneqq C_0\left(t^2+(1+|z|^2)^2\right)^{-\frac{Q-2}{4}}    
\end{align}
satisfies \eqref{hom_eqn} and $C_0$ is a positive normalization constant such that
\begin{align}\label{sobolev-const}
    \int_{\mathbb{H}^n}\left|\nabla_{\mathbb{H}^n}U(\xi)\right|^2 d\xi = \int_{\mathbb{H}^n}\left|U(\xi)\right|^{2^{*}} d\xi \, = \mathcal{S}_{Q}^{Q/2},
\end{align}
where $\mathcal{S}_{Q}$ denotes the best constant in the Folland-Stein-Sobolev embedding.
\end{theorem}
Critical equations on $\mathbb{H}^n$ and related Carnot groups, particularly subelliptic analogues of the Br\'ezis-Nirenberg problem, have been the subject of extensive study. On bounded domains, Garofalo and Lanconelli \cite{GL92} established sharp existence and nonexistence results for semilinear critical equations, revealing a threshold phenomenon analogous to the classical Euclidean case. Furthermore, Citti and Uguzzoni \cite{CU01} demonstrated that the existence of solutions is deeply influenced by the topology of the domain. On the entire space $\mathbb{H}^n$, the Jerison--Lee classification dictates that positive solutions are strictly confined to the bubble family, a rigidity further reinforced by recent Liouville-type theorems \cite{PV25, CLMR25}. From a variational perspective, a global compactness theorem in the spirit of Struwe was recently established on some bounded domains of $\mathbb{H}^n$ by Palatucci, Piccinini, and Temperini \cite{PPT25}; for an overview of related works in this direction, we refer the reader to \cite{PPT25, Gam01, GY01} and the references therein. Additionally, critical singular problems on Carnot groups have been investigated in \cite{BGV26}.

\subsection*{The problem and main results}
Motivated by the rigidity of the homogeneous case, we study the following \emph{inhomogeneous} critical subelliptic problem on $\mathbb{H}^n$:
\begin{equation}\label{eqn_pert_Hn}\tag{$\wp$}
\begin{cases}
    \mathcal{L}_{\mathbb{H}^n}u = a(\xi)|u|^{2^{\star}-2}u + f(\xi)
    & \text{in }\mathbb{H}^n,\\
    u>0,\quad u\in S^1(\mathbb{H}^n).
\end{cases}
\end{equation}
When $a\equiv 1$ and $f\equiv 0$, \eqref{eqn_pert_Hn} reduces to the standard CR Yamabe equation \eqref{hom_eqn}. However, the inclusion of the nonhomogeneous term $f\not\equiv 0$ breaks the scale and translation invariances of the space, opening the possibility for new positive solutions beyond the Jerison--Lee family. Consequently, \eqref{eqn_pert_Hn} acts as an inhomogeneous perturbation of the homogeneous critical equation, serving as the direct subelliptic analogue of the nonhomogeneous Euclidean problems studied in \cite{Tar92, BP20}. On bounded domains, compactness can only be lost through concentration, and Tarantello's approach also works in the subelliptic setting; see \cite{Loi20}. On the whole space $\mathbb{H}^n$, however, compactness can also be lost through translations to infinity and through dilations with $\lambda\to\infty$. This is the main reason why we need a full profile decomposition in the present setting.

Throughout the article, we impose the following conditions on the potential $a$ and the nonhomogeneous term $f$:
\begin{assumption}[\textbf{A}]\label{assum_main}
$a\in L^\infty(\mathbb{H}^n)\cap C(\mathbb{H}^n)$ with $\inf_{\mathbb{H}^n}a>0$ and $\lim_{\|\xi\|\to\infty}a(\xi) = 1$.
\end{assumption}
\begin{assumption}[\textbf{F}]\label{cond_on_f}
$f\not\equiv 0$ is a nonnegative functional in $(S^1(\mathbb{H}^n))'$, meaning ${}_{(S^1)'}\langle f,\varphi\rangle_{S^1}\geq 0$ for all $\varphi\geq 0$ in $S^1(\mathbb{H}^n)$.
\end{assumption}
\begin{definition}
A function $u\in S^1(\mathbb{H}^n)$ is said to be a positive (weak) solution of \eqref{eqn_pert_Hn} if $u>0$ almost everywhere on $\hn$ and it satisfies the following variational identity
\[
    \int_{\mathbb{H}^n}\nabla_{\mathbb{H}^n}u\cdot \nabla_{\mathbb{H}^n}\varphi\,d\xi
    = \int_{\mathbb{H}^n}a(\xi)|u|^{2^{\star}-2}u\,\varphi\,d\xi
    + {}_{(S^1)'}\langle f,\varphi\rangle_{S^1},
\]
for all $\varphi\in S^1(\mathbb{H}^n)$.
\end{definition}

From a variational perspective, these weak solutions arise naturally as critical points of an underlying energy. Specifically, solutions to \eqref{eqn_pert_Hn} correspond exactly to the positive critical points of the Euler--Lagrange functional $\mathcal{E}_{a,f}:S^1(\mathbb{H}^n)\to\mathbb{R}$, defined by
\begin{equation}\label{def_E_K,f}
    \mathcal{E}_{a,f}(u)
    := \frac{1}{2}\int_{\mathbb{H}^n}|\nabla_{\mathbb{H}^n}u|^2\,d\xi
    - \frac{1}{2^{\star}}\int_{\mathbb{H}^n}a(\xi)|u|^{2^{\star}}\,d\xi
    - {}_{(S^1)'}\langle f,u\rangle_{S^1}.
\end{equation}

To establish the existence of such critical points, one must analyze the compactness properties of this functional. This is typically done by studying sequences along which the energy is bounded and its derivative vanishes.

\begin{definition}
A sequence $\{u_k\} \subset S^{1}(\mathbb{H}^n)$ is said to be a \textit{Palais--Smale sequence} for $\mathcal{E}_{a,f}$ at level $c \in \mathbb{R}$ (denoted as a $(PS)_c$ sequence) if
\begin{equation*}
    \mathcal{E}_{a,f}(u_k) \to c \quad \text{and} \quad  \mathcal{E}_{a,f}'(u_k) \to 0 \quad\text{in } (S^1(\mathbb{H}^n))' \quad \text{as } k \to \infty.
\end{equation*}
The functional $\mathcal{E}_{a,f}$ is said to satisfy the \textit{$(PS)_c$ condition} if every Palais-Smale sequence at level $c$ is relatively compact in $S^1(\hn)$. 
\end{definition}
 Since the embedding $S^1(\mathbb{H}^n)\hookrightarrow L^{2^\star}(\mathbb{H}^n)$ is invariant under left translations and non-isotropic dilations, one cannot expect $\mathcal{E}_{a,f}$ to satisfy the $(PS)_c$ condition in general. For bounded sequences in $S^1(\mathbb{H}^n)$, the defect of compactness of this embedding was described by Benameur \cite{Ben08} in the spirit of G\'erard \cite{Ger98}, while Palatucci, Piccinini and Temperini \cite{PPT25} obtained a Struwe-type global compactness result on bounded domains of $\mathbb{H}^n$. However, none of these results deals with Palais--Smale sequences of the nonhomogeneous functional $\mathcal{E}_{a,f}$ on the whole space. Our first main result fills this gap and provides a decomposition for every $(PS)_c$ sequence of $\mathcal{E}_{a,f}$ under Assumptions~\hyperref[assum_main]{\textbf{(A)}} and~\hyperref[cond_on_f]{\textbf{(F)}}.

\begin{theorem}\label{thm_PS_decomp}
Let $\{u_k\}\subset S^1(\mathbb{H}^n)$ be a $(PS)_c$ sequence for
$\mathcal{E}_{a,f}$ at level $c\in\mathbb{R}$. Then, up to a subsequence
(still denoted by $\{u_k\}$), there exists an integer $m\geq 0$ such that the following holds

\begin{enumerate}
    \item\textit{(Behavior of parameters.)} For each $j=1,\ldots,m$, there exist sequences of points $\{\xi_k^{(j)}\}\subset\mathbb{H}^n$ and
scales $\{\lambda_k^{(j)}\}\subset(0,\infty)$, such that 
    \begin{equation}\label{eq:param_behavior}
        |\log\lambda_k^{(j)}| + \|\xi_k^{(j)}\|\to\infty, \quad \text{as } k\to\infty.
     \end{equation} 
    
    \item \textit{(Existence of weak solutions.)} There exists a weak solution
$u^{(0)}\in S^1(\mathbb{H}^n)$ to \eqref{eqn_pert_Hn} (without the positivity condition), and $m$ nontrivial weak solutions $u^{(1)},\ldots,u^{(m)}\in S^1(\mathbb{H}^n)$ to the limiting equations
\begin{equation}\label{eq:limiting_eq_j}
    \mathcal{L}_{\hn}u^{(j)} = a_\infty^{(j)}\, |u^{(j)}|^{2^{\star}-2}u^{(j)}
    \quad\text{in }\mathbb{H}^n,\quad j=1,\ldots,m,
\end{equation}
where
\begin{equation}\label{eq:a_infty}
    a_\infty^{(j)}\coloneqq
    \begin{cases}
        a(\xi^{(j)}) & \text{if } \lambda_k^{(j)}\to 0 \text{ and } \xi_k^{(j)}\to\xi^{(j)}\in\mathbb{H}^n,\\
        1 & \text{otherwise}.
    \end{cases}
\end{equation}
    \item\label{it_profile_PS_thm} \textit{(Profile decomposition.)} The sequence decomposes as
    \begin{equation}\label{eq_seq_decomp}
        u_k - \left(u^{(0)}
        +\sum_{j=1}^m
        D_{\xi_k^{(j)},\lambda_k^{(j)}}u^{(j)}\right)
        \to 0
        \quad\text{strongly in }S^1(\mathbb{H}^n).
    \end{equation}

    \item\label{ite_aymp_ortho}\textit{(Asymptotic orthogonality.)} The distinct bubble profiles do not interact asymptotically, i.e., for $i\neq j$,  
    \begin{equation}\label{eq:orth}
        \left|\log\frac{\lambda_k^{(i)}}{\lambda_k^{(j)}}\right|
        +\left\|\delta_{1/\lambda_k^{(j)}}
        \!\left((\xi_k^{(j)})^{-1}\circ\xi_k^{(i)}\right)\right\|
        \to\infty,\quad\text{as }k\to\infty.
    \end{equation}

    \item\label{item_en_dec_PS_thm} \textit{(Energy decoupling.)} We have
    \begin{equation}\label{eq:energy_decomp}
        \mathcal{E}_{a,f}(u_k)
        \to\mathcal{E}_{a,f}(u^{(0)})
        +\sum_{j=1}^m
        \mathcal{E}_{a_\infty^{(j)},0}(u^{(j)}),\quad\text{as }k\to\infty,
    \end{equation}
    where $\mathcal{E}_{a_\infty^{(j)},0}$ denotes the functional \eqref{def_E_K,f}
    with $a$ replaced by the constant $a_\infty^{(j)}$ and $f\equiv 0$.
\end{enumerate}  
Furthermore, if $u_k\geq 0$ a.e.\ for all $k$, then $u^{(0)}\geq 0$ and
$u^{(j)}\geq 0$ for all $j=1,\ldots,m$. In this case, each $u^{(j)}$ is a
nontrivial nonnegative weak solution of \eqref{eq:limiting_eq_j}, and hence
$u^{(j)}>0$ a.e.\ in $\mathbb{H}^n$ by the strong maximum principle. Since $(a_\infty^{(j)})^{\frac{Q-2}{4}}u^{(j)}$ is then a positive solution of \eqref{hom_eqn}, we obtain from Theorem~\ref{thm_JL_bubble} that
\[
    u^{(j)} = (a_\infty^{(j)})^{-\frac{Q-2}{4}}D_{\eta^{(j)},\lambda^{(j)}}U
    \quad\text{for some }\eta^{(j)}\in\mathbb{H}^n,\;\lambda^{(j)}>0,
\]
where $U$ is the standard Jerison--Lee bubble \eqref{eq_J_L_extreme}.
\end{theorem}
In view of Theorem~\ref{thm_PS_decomp} and Theorem~\ref{thm_JL_bubble}, a nonnegative Palais--Smale sequence can lose compactness only at certain specific energy levels. It is then natural to ask whether compactness indeed fails at each of these levels. Our second main result shows that this is the case, and hence the description in Theorem~\ref{thm_PS_decomp} is sharp. To state it, let us denote by $\Lambda_{a,f}$ the set of all levels
\[
  \Lambda_{a,f}:=\left\{c \in \R : c=  \mathcal{E}_{a,f}(u^{(0)})+\frac{1}{Q}\,\mathcal{S}_Q^{Q/2}\sum_{j=1}^m\alpha_j^{-\frac{Q-2}{2}} \right\},
\]
where $m\geq1$, $u^{(0)}\geq 0$ is a weak solution of \eqref{eqn_pert_Hn} (without the positivity condition), and $\alpha_1,\ldots,\alpha_m\in a(\mathbb{H}^n)\cup\{1\}$; here $a(\mathbb{H}^n)=\{a(\xi):\xi\in\mathbb{H}^n\}$ denotes the range of $a$. We note from \eqref{eq:a_infty} that these are precisely the possible values of the constants $a_\infty^{(j)}$ in Theorem~\ref{thm_PS_decomp}: $a(\xi^{(j)})$ when a bubble concentrates at $\xi^{(j)}$, and $1$ when a bubble escapes to infinity or spreads out.

\begin{theorem}\label{thm_exact_levels}
Every nonnegative $(PS)_c$ sequence of $\mathcal{E}_{a,f}$ is relatively compact in $S^1(\mathbb{H}^n)$ if and only if $c\notin\Lambda_{a,f}$. Moreover, for every $c\in\Lambda_{a,f}$, there exists a nonnegative $(PS)_c$ sequence which converges weakly, but not strongly, in $S^1(\mathbb{H}^n)$.
\end{theorem}

In particular, when $a\equiv 1$, the levels at which compactness fails are exactly $\mathcal{E}_{1,f}(u^{(0)})+\frac{m}{Q}\mathcal{S}_Q^{Q/2}$, $m\geq1$. We prove the second part of Theorem~\ref{thm_exact_levels} by constructing such sequences explicitly; see Theorem~\ref{thm_sharp_levels}.

We now turn to the existence of positive solutions of \eqref{eqn_pert_Hn}. If $f$ satisfies \hyperref[cond_on_f]{\textbf{(F)}} and is suitably small, one can obtain a first positive solution with negative energy as a local minimizer of the energy. Our main interest is in finding a \textit{second} positive solution, which we obtain by a mountain-pass argument. For this purpose, we need the following assumption on the behaviour of $a$ near one of its maximum points, where $d$ denotes the Kor\'anyi distance (see Section~\ref{prelim}).
\begin{assumption}[\textbf{M}]\label{assum_max}
The coefficient $a$ attains its maximum at some $\zeta\in\mathbb{H}^n$, and there exist $C,r>0$ and $\beta>\frac{Q-2}{2}$ such that $a(\xi)\geq\|a\|_{L^\infty(\mathbb{H}^n)}-C\,d(\xi,\zeta)^{\beta}$ whenever $d(\xi,\zeta)<r$.
\end{assumption}
\begin{theorem}\label{thm_multi_pos_sol}
Assume \hyperref[assum_main]{\textbf{(A)}}, \hyperref[cond_on_f]{\textbf{(F)}} and \hyperref[assum_max]{\textbf{(M)}}, and that
\begin{align}\label{eq:f_small} 
 \|f\|_{(S^1)'} < C_0\, \|a\|_{L^\infty(\mathbb{H}^n)}^{-\frac{Q-2}{4}}\mathcal{S}_Q^{Q/4}, \qquad \text{where} \quad C_0 = \frac{4}{Q+2}(2^\star-1)^{-\frac{Q-2}{4}}, 
\end{align}
and $\mathcal{S}_Q$ denotes the best Folland--Stein constant given by \eqref{sobolev-const}. Then \eqref{eqn_pert_Hn} admits at least two distinct positive solutions.  
\end{theorem}
Let us make a few comments on Assumption~\hyperref[assum_max]{\textbf{(M)}}. It clearly holds for $a\equiv 1$, and therefore Theorem~\ref{thm_multi_pos_sol} contains the constant coefficient case. It also holds if $a$ is constant near one of its maximum points. Moreover, there are smooth non-constant coefficients satisfying \hyperref[assum_max]{\textbf{(M)}}, for instance $a(z,t)=1+\big(1+(|z|^4+t^2)^k\big)^{-1}$ with $k\in\mathbb{N}$, $k>\frac n4$; indeed, here $\zeta=e$, $\|a\|_{L^\infty}=2$ and $\|a\|_{L^\infty}-a(\xi)\leq d(\xi,e)^{4k}$. We note, however, that for $n\geq2$, \hyperref[assum_max]{\textbf{(M)}} does not allow maxima which are nondegenerate in the horizontal directions. The constant $\|a\|_{L^\infty}$ appears in \eqref{eq:f_small} because, by Theorem~\ref{thm_exact_levels}, the lowest noncompactness level of the form $\mathcal{E}_{a,f}(u_0)+c$ with $c>0$ is $\mathcal{E}_{a,f}(u_0)+\frac{1}{Q}\|a\|_{L^\infty}^{-\frac{Q-2}{2}}\mathcal{S}_Q^{Q/2}$, where $u_0$ denotes the first solution. Finally, the condition $\beta>\frac{Q-2}{2}$ ensures that, for a bubble concentrating at $\zeta$ with scale $\varepsilon$, the gain of order $\varepsilon^{\frac{Q-2}{2}}$ coming from its interaction with $u_0$ dominates the loss of order $\varepsilon^{\beta}$ caused by the variation of $a$ near $\zeta$.

\begin{remark}\label{rem_general_group}
Two features of Theorems~\ref{thm_PS_decomp} and~\ref{thm_multi_pos_sol} rely specifically on the Heisenberg setting $\mathbb{H}^n$. First, combining Theorem~\ref{thm_PS_decomp} with the classification in Theorem~\ref{thm_JL_bubble} identifies every nonnegative profile as an explicit Jerison--Lee bubble, which makes the noncompactness levels in Theorem~\ref{thm_exact_levels} explicit. Second, because the sharp constant $\mathcal{S}_Q$ is known explicitly, so is the smallness threshold in \eqref{eq:f_small}. On a general Carnot group, the best Folland--Stein constant is attained \cite{Vas06}, yet neither its explicit value nor a classification of positive limiting solutions is known, preventing such explicit descriptions. Conversely, our concentration-detection mechanism extends beyond $\mathbb{H}^n$: Theorem~\ref{thm_hardy_inter} holds on general stratified groups (Remark~\ref{rem_strat}), and the bubble energy lower bound relies solely on the Folland--Stein inequality (Remark~\ref{rem_why_not_fractional}).
\end{remark}
\subsection*{Strategy of proof and key contributions}
Let us briefly describe the main steps of our argument, which rests on three ingredients.

\smallskip
\noindent\textbf{1. Improved Folland--Stein--Sobolev inequality.} 
A central difficulty in critical problems is detecting the concentration of Palais--Smale sequences. In the Euclidean setting, this is typically handled via Lions' concentration-compactness principle \cite{Lio84a, Lio84b}. On the Heisenberg group, we develop an alternative mechanism by establishing a new improved Folland--Stein--Sobolev inequality of Morrey type. Specifically, we first prove a novel embedding of Morrey spaces into Besov spaces on $\mathbb{H}^n$:
\begin{equation}\label{Besov-Morrey-embedd}
    \mathcal{M}^{1,\alpha}(\mathbb{H}^n) \hookrightarrow \dot{B}^{-\alpha}_{\infty,\infty}(\mathbb{H}^n) \quad\text{for } 0<\alpha<Q,
\end{equation}
(see Lemma~\ref{lem_Mor<Bes}). By coupling \eqref{Besov-Morrey-embedd} with Chamorro's improved Sobolev inequality in Besov spaces \cite{C11}, we establish the following theorem, which serves as the subelliptic counterpart to the Euclidean results of Palatucci and Pisante \cite{PP14}.

\begin{theorem}\label{thm_hardy_inter}
Let $1\leq p<Q$ and $p^*=Qp/(Q-p)$. There exists $C=C(Q,p)>0$ such that for any $p/p^*\leq\theta<1$ and any $1\leq r<p^*$, 
\begin{equation}\label{eq_impro_sobolev}
    \|u\|_{L^{p^*}(\mathbb{H}^n)} \leq C\|\nabla_{\mathbb{H}^n}u\|_{L^p(\mathbb{H}^n)}^\theta \|u\|_{\mathcal{M}^{r,r(Q-p)/p}(\mathbb{H}^n)}^{1-\theta} \quad\text{for all } u\in  S^{1,p}(\mathbb{H}^n).
\end{equation}
\end{theorem}
For $p=2$, inequality \eqref{eq_impro_sobolev} becomes our primary tool for detecting concentration. If a bounded sequence in $S^1(\mathbb{H}^n)$ does not vanish in $L^{2^{\star}}(\mathbb{H}^n)$, \eqref{eq_impro_sobolev} forces its Morrey norm to be bounded strictly away from zero. This localizes the concentration to a specific Heisenberg ball, allowing us to extract a nonzero weak profile via appropriate left-translations and non-isotropic dilations.

\smallskip
\noindent\textbf{2. Palais--Smale profile decomposition on $\mathbb{H}^n$.} 
By iteratively applying this concentration detection mechanism, we establish the full Palais--Smale profile decomposition (Theorem~\ref{thm_PS_decomp}). While profile decompositions have a rich Euclidean history---from Struwe \cite{Str84} and Solimini \cite{Sol95} to the wavelet approaches of G\'erard \cite{Ger98} and Jaffard \cite{Jaf99}---our framework extends the homogeneous $\mathbb{H}^n$ compactness results of \cite{PPT25} to the {inhomogeneous, variable-coefficient} setting. Two main difficulties arise compared with the Euclidean case. First, because the group law is non-commutative and the dilations $\delta_\lambda$ are non-isotropic, the scaling and translation parameters of the bubbles must be tracked carefully. We handle this by formulating the $\mathbb{H}^n$ analogue of G\'erard's characterization of asymptotically orthogonal translation--dilation sequences (Lemma~\ref{pp-lemma-3}). Second, passing to the limit in the coefficient $a$ along bubble sequences yields the limiting constants $a_\infty^{(j)}$, and we apply the weak Harnack inequality \cite{BLU07} to deduce that nontrivial nonnegative profiles are strictly positive. Conversely, we prove that compactness fails at every energy level permitted by the decomposition, which precisely characterizes the compactness thresholds for nonnegative sequences (Theorem~\ref{thm_exact_levels}).

\smallskip
\noindent\textbf{3. Multiplicity of positive solutions.}
Finally, we apply the above results to \eqref{eqn_pert_Hn}. The first solution $u_0$ is obtained as a local minimizer of the energy with negative energy, under the smallness condition on $\|f\|_{(S^1)'}$ (cf.\ \cite{Tar92},  \cite[Theorem 1.1]{BP20}). For the second solution, we use the mountain-pass theorem, which gives us a $(PS)_{c^*}$ sequence. In view of Theorems~\ref{thm_PS_decomp} and~\ref{thm_exact_levels}, compactness of this sequence can only fail if $c^*\geq\mathcal{E}_{a,f}(u_0)+\frac{1}{Q}\|a\|_{L^\infty}^{-\frac{Q-2}{2}}\mathcal{S}_Q^{Q/2}$, or if the sequence converges to a critical point different from $u_0$, in which case we already have a second solution. Using a bubble concentrating at a maximum point of $a$, we keep $c^*$ strictly below this level, and this gives us a second positive solution.

\subsection*{Structure of the article} In Section~\ref{prelim}, we collect some preliminaries. Section~\ref{improved sobolev} is devoted to the proof of Theorem~\ref{thm_hardy_inter}. In Section~\ref{ps section}, we prove Theorems~\ref{thm_PS_decomp} and~\ref{thm_exact_levels}, and finally, in Section~\ref{mul sol cry}, we prove Theorem~\ref{thm_multi_pos_sol}.

\section{Preliminaries}\label{prelim}
Let $\mathbb{H}^{n}= \mathbb{C}^n\times \mathbb{R} \simeq \mathbb{R}^{2n+1}$ be the $(2n+1)$-dimensional Heisenberg group. Any element $\xi\in\mathbb{H}^n$ can be denoted by
\[
    \xi = (z,t) = (x+iy, t) \simeq (x_1,\ldots,x_n, y_1,\ldots,y_n, t).
\]
For $\xi=(x,y,t)$ and $\eta=(x',y',s)\in\mathbb{H}^n$, the non-commutative group law is given by
\[
    \xi\circ\eta
    =\!\left(x+x',\;y+y',\;t+s+\tfrac{1}{2}
    \sum_{i=1}^n(y_ix_i'-x_iy_i')\right).
\]
The inverse of an element $\xi=(z,t)$ is given by $\xi^{-1}=(-z,-t)$, and the identity element is $e=(0,0)$. For any $\eta \in \mathbb{H}^n$, the left-translation operator is defined as $\tau_{\eta}(\xi) := \eta^{-1}\circ\xi$ for all $\xi\in\mathbb{H}^n$.

The Heisenberg Lie algebra $\mathfrak{h}^n$ is spanned by the $2n+1$ left-invariant vector fields
\begin{align*}
    X_j &= \frac{\partial}{\partial x_j}+\frac{1}{2}y_j\frac{\partial}{\partial t},
    \quad j=1,\ldots,n,\\
    X_j &= \frac{\partial}{\partial y_{j-n}}-\frac{1}{2}x_{j-n}\frac{\partial}{\partial t},
    \quad j=n+1,\ldots,2n,
\end{align*}
and $T=\frac{\partial}{\partial t}$. The sub-Laplacian on the Heisenberg group is the second-order differential operator $\mathcal{L}_{\mathbb{H}^n}=-\sum_{j=1}^{2n}X_j^2$, which takes the explicit form
\begin{align}\label{eq_def_sub_lap}
    \mathcal{L}_{\mathbb{H}^n} = -\Delta_{\mathbb{C}^n}
    -\frac{1}{4}|z|^2\frac{\partial^2}{\partial t^2}
    +\sum_{j=1}^n\!\left(x_j\frac{\partial}{\partial y_j}
    -y_j\frac{\partial}{\partial x_j}\right)\frac{\partial}{\partial t}.
\end{align}
We denote the horizontal gradient by $\nabla_{\mathbb{H}^n}=(X_1,\ldots,X_{2n})$.

The homogeneous (Kor{\'a}nyi) norm on $\mathbb{H}^n$ is defined as
\[
    \|\xi\|=\|(z,t)\|=\bigl(|z|^4+t^2\bigr)^{1/4},
\]
and the associated distance is $d(\xi,\eta)=\|\xi^{-1}\circ\eta\|$. The non-isotropic dilation $\delta_{\lambda} \colon\mathbb{H}^n\to\mathbb{H}^n$ for $\lambda>0$ is defined by $\delta_{\lambda}(z,t)=(\lambda z,\lambda^2t)$ which satisfies $\|\delta_{\lambda}(\xi)\| = \lambda\|\xi\|$. The Haar measure on $\mathbb{H}^n$ coincides with the standard $(2n+1)$-dimensional Lebesgue measure. For $\xi \in \hn$, we write that $B(\xi, r)$ is a ball on $\hn$ centered at $\xi$ and radius $r>0$ and given by $B(\xi, r)=\{\eta \in \hn : d(\eta, \xi)<r\}$. Moreover, the Lebesgue measure of the ball $B(\xi, r)$ satisfies $|B(\xi, r)|=r^{Q}|B(\xi,1)|$, here $Q=2n+2$ is known as the homogeneous dimension of $\hn$.

For $f,g\in L^1(\mathbb{H}^n)$, the group convolution is defined by
\[
    (f*g)(\xi)
    =\int_{\mathbb{H}^n}f(\xi\circ\eta^{-1})\,g(\eta)\,d\eta.
\]

The operator $\mathcal{L}_{\mathbb{H}^n}$ generates a heat contraction semigroup $\{e^{-t\mathcal{L}_{\mathbb{H}^n}}\}_{t>0}$, which acts by convolution, $e^{-t\mathcal{L}_{\mathbb{H}^n}}f=f\ast p_t$, with a heat kernel $p_t$ (see \cite[Chapter~1]{Thanbook} for its explicit form). For all $t,\lambda>0$ and $\xi\in\mathbb{H}^n$, it satisfies the scaling property $  p_t(\delta_{\lambda}(\xi)) = \lambda^{-Q} p_{\frac{t}{\lambda^2}}(\xi).$
Furthermore, the kernel $p_t$ satisfies the following conservative property together with the standard Gaussian-type estimates (see \cite[pages~84-85]{Thanbook} for instance)
\begin{align}\label{heat-kernel-prop}
    \int_{\mathbb{H}^n} p_t(\xi) \,d\xi = 1 \qquad \text{and} \qquad |p_{t}(\xi)| \leq C_0 t^{-Q/2} e^{-C_1\frac{\|\xi\|^2}{t}},
\end{align}
for some universal positive constants $C_0$ and $C_1$. 

Using the horizontal vector fields $X_j$, we define the homogeneous Folland--Stein--Sobolev space $S^{1,p}(\mathbb{H}^n)$ as the completion of $C_c^\infty(\mathbb{H}^n)$ with respect to the norm
\begin{align}\label{hom-sobl}
    \| u\|_{S^{1,p}(\mathbb{H}^n)} :=  \left(\int_{\mathbb{H}^n}|\nabla_{\mathbb{H}^n} u(\xi )|^p\, d\xi \right)^{\frac1p}.
\end{align}
In particular, for $p=2$, the space $S^{1,2}(\mathbb{H}^n)$ is a Hilbert space equipped with the inner product
\begin{align*}
    \langle u,v\rangle_{S^{1,2}(\mathbb{H}^n)} = \int_{\mathbb{H}^n} \nabla_{\mathbb{H}^n} u(\xi ) \cdot \nabla_{\mathbb{H}^n} v(\xi ) \, d\xi \qquad \text{for all } u,\, v \in S^{1,2}(\mathbb{H}^n).
\end{align*}

Next, we introduce two important function spaces that play a crucial role in the proofs of Theorems~\ref{thm_PS_decomp} and \ref{thm_hardy_inter}.
\begin{definition}
For $\alpha>0$, the homogeneous Besov space $\dot{B}^{-\alpha}_{\infty, \infty}(\mathbb{H}^n)$ is defined as the space of tempered distributions $f \in \mathcal{S}'(\mathbb{H}^n)$ such that
\begin{align*}
    \|f\|_{\dot{B}^{-\alpha}_{\infty, \infty}(\mathbb{H}^n)} = \sup_{t>0} \, t^{\alpha/2} \|e^{-t\mathcal{L}_{\mathbb{H}^n}}f\|_{L^{\infty}(\mathbb{H}^n)} < \infty.
\end{align*} 
\end{definition}

\begin{definition}\label{def_mor}
   Let $\gamma \in [0, Q]$ and $q \in [1, \infty)$. A measurable function $f$ is said to belong to the Morrey space $\mathcal{M}^{q,\gamma}(\mathbb{H}^n)$ if 
   \begin{align*}
       \|f\|_{\mathcal{M}^{q,\gamma}(\mathbb{H}^n)}^{q} := \sup_{\substack{R>0\\ \xi \in \mathbb{H}^n}}  \frac{R^{\gamma}}{|B(\xi,R)|} \int_{B(\xi, R)} |f(\eta)|^{q} d\eta < \infty.
   \end{align*}
   \end{definition}
   
  The parameter $\gamma$ interpolates between the Lebesgue and essentially bounded regimes. Indeed, when $\gamma=Q$, $\mathcal{M}^{q,Q}(\mathbb{H}^n)$ coincides with $L^q(\mathbb{H}^n)$, while for $\gamma=0$, $\mathcal{M}^{q,0}(\mathbb{H}^n)$ coincides with $L^\infty(\mathbb{H}^n)$. The Morrey spaces introduced above will be used to establish an embedding into negative-order Besov spaces in the next section.

\section{An improved Folland--Stein--Sobolev inequality}\label{improved sobolev}
In this section, we prove the improved interpolated version of the Folland--Stein--Sobolev inequality, i.e., Theorem~\ref{thm_hardy_inter}. To prove the theorem, we first recall an improved Sobolev inequality in terms of Besov norms studied in \cite[Theorem~1]{C11}.
\begin{theorem}[\cite{C11}]\label{th-chamar}
    For every $1\leq p<q<\infty$ and every function $u\in S^{1,p}(\mathbb{H}^n)$, there holds
    \begin{align}\label{dn-l-1}
        \|u\|_{L^q(\mathbb{H}^n)}\leq C \|\nabla_{\mathbb{H}^n} u \|_{L^p(\mathbb{H}^n)}^\theta\|u\|_{\dot{B}_{\infty,\infty}^{\theta/(\theta-1)}(\mathbb{H}^n)}^{1-\theta}\, ,
    \end{align}
    where $\theta=p/q$ and the positive constant $C$ depends only on $p$, $q$, and $n$.
\end{theorem}

We next state two technical lemmas establishing basic continuous embeddings for Morrey spaces. Their proofs follow directly from the definition of the Morrey norm and the application of H\"older's inequality over Heisenberg balls. We omit the proof for the sake of brevity.
\begin{lemma}
  For any $\gamma \in (0, Q)$ and $q \in (1, \infty)$, there exists a positive constant $C$ depending on $Q$, $q$, and $\gamma$ such that for all $u\in \mathcal{M}^{q,\gamma}(\mathbb{H}^n)$, there holds
  \begin{align}\label{dn-l-2}
      \|u\|_{\mathcal{M}^{1,\gamma/q}(\mathbb{H}^n)}\leq C \|u\|_{\mathcal{M}^{q,\gamma}(\mathbb{H}^n)}.
  \end{align}
\end{lemma}

\begin{lemma}\label{morey-sblv}
Let $1\leq p<Q$. Then for any $1\leq r < p^*= Qp /(Q-p)$, there holds the continuous embedding
    \begin{align*}
    L^{p^*}(\mathbb{H}^n)\hookrightarrow \mathcal{M}^{r, r(Q-p)/p}(\mathbb{H}^n).
\end{align*}
\end{lemma}

In the following lemma, we establish a crucial embedding between Morrey spaces and Besov spaces of negative orders on the Heisenberg group.

\begin{lemma}\label{lem_Mor<Bes}
    Assume $\alpha\in (0, Q)$. Then every $u\in \mathcal{M}^{1,\alpha}(\mathbb{H}^n)$ is a tempered distribution, and there exists a constant $C=C(Q,\alpha)>0$ such that 
    \begin{align}\label{dn-l-3}
       \|u\|_{\dot{B}^{-\alpha}_{\infty, \infty}(\mathbb{H}^n)}\leq C \|u\|_{\mathcal{M}^{1,\alpha}(\mathbb{H}^n)}
    \end{align}
    for all $u\in \mathcal{M}^{1,\alpha}(\mathbb{H}^n)$.
\end{lemma}
\begin{proof}
  Let $u\in \mathcal{M}^{1,\alpha}(\mathbb{H}^n)$. By definition, for any $R>0$ and $\xi \in \mathbb{H}^n$, we have
  \begin{align}\label{ball-est}
      \int_{B(\xi,R)}|u(\eta)|\,d\eta \leq  C R^{Q-\alpha}\|u\|_{\mathcal{M}^{1,\alpha}(\mathbb{H}^n)}.
  \end{align}
  First, we verify that $u \in \mathcal{S}'(\mathbb{H}^n)$, that is $u$ defines a tempered distribution. It suffices to show that there exists an integer $m \geq Q$ such that $\int_{\mathbb{H}^n} |u(\eta)|(1+\|\eta\|^m)^{-1} d\eta < \infty$. Using a dyadic decomposition of $\mathbb{H}^n$ centered at the identity $e$, we define the annuli
  \begin{align*}
      A_k:=\left\{\eta\in\mathbb{H}^n\,:\, 2^{-k} <\|\eta\| \leq 2^{-k+1} \right\}, \quad k \in \mathbb{Z}.
  \end{align*}
  Using \eqref{ball-est} over the balls $B(e, 2^{-k+1})$, we estimate
  \begin{align*}
      \int_{\mathbb{H}^n} \frac{|u(\eta)|}{1+\|\eta\|^m}\,d\eta 
      =\sum_{k\in\mathbb{Z}}\int_{A_k} \frac{|u(\eta)|}{1+\|\eta\|^m}\,d\eta
      \leq \sum_{k\in\mathbb{Z}} \frac{1}{1+2^{-mk}}\int_{B(e, 2^{-k+1})}|u(\eta)|\,d\eta 
      \leq C \|u\|_{\mathcal{M}^{1,\alpha}(\mathbb{H}^n)} \sum_{k\in\mathbb{Z}} \frac{ 2^{-k(Q-\alpha)} }{1+2^{-mk}}.
  \end{align*}
  Since $\alpha < Q$, the sum converges for $k \to \infty$. For $k \to -\infty$, the general term behaves like $2^{k(m-Q+\alpha)}$, which converges provided we choose $m > Q-\alpha$. Thus, $u$ is a tempered distribution. 
  
  \smallskip
  
  Next, observe that for any $R>0$, the non-negativity of the heat kernel (refer \cite[pages~84-85]{Thanbook}) implies that for all $\xi \in \mathbb{H}^n$,
  \begin{align*}
      |(e^{-R^2\mathcal{L}_{\mathbb{H}^n}}u)(\xi)| = |(u \ast p_{R^2})(\xi)|\leq \int_{\mathbb{H}^{n}}| u\left(\xi \circ \eta^{-1}\right)|\, p_{R^2}(\eta)\,d\eta = (e^{-R^2\mathcal{L}_{\mathbb{H}^n}}|u|)(\xi).
  \end{align*}
  For $\xi\in\mathbb{H}^n$ and $R>0$, let us define
  \begin{align*}
     G_{R}(\xi) := (e^{-R^2\mathcal{L}_{\mathbb{H}^n}}|u|)(\xi) = \int_{\mathbb{H}^{n}} |u(\eta)|\, p_{R^2}\left( \eta^{-1}\circ \xi\right) \, d\eta.
  \end{align*}
  We claim that  
  \begin{align}\label{inq_G_R<}
       G_{R}(\xi) \leq C \|u\|_{\mathcal{M}^{1,\alpha}(\mathbb{H}^n)} R^{-\alpha},
  \end{align} 
  for all $\xi \in \mathbb{H}^n$, where the constant $C$ is independent of $\xi$. To prove this, we perform a dyadic decomposition scaled by $R$:
  \begin{align*}
      \tilde{A}_k:=\left\{\eta\in\mathbb{H}^n\,:\, 2^{-k}R <\|\eta^{-1}\circ \xi\| \leq 2^{-k+1}R \right\}, \quad k \in \mathbb{Z}.
  \end{align*}
 Using the upper bound of the heat kernel \eqref{heat-kernel-prop}, we have
  \begin{align*}
      G_{R}(\xi) = \sum_{k\in \mathbb{Z}} \int_{\tilde{A}_k} |u(\eta)|\, p_{R^2}\left(\eta^{-1}\circ \xi\right)\,d\eta
      \leq C \sum_{k\in \mathbb{Z}} \int_{\tilde{A}_k} |u(\eta)|\, R^{-Q} e^{-C_1\frac{\|\eta^{-1}\circ \xi\|^2}{R^2}}\,d\eta.
  \end{align*}
  Since $\|\eta^{-1}\circ \xi\| > 2^{-k} R$ on $\tilde{A}_k$, we obtain
  \begin{equation*}
      G_{R}(\xi) \leq C R^{-Q}  \sum_{k\in \mathbb{Z}}  e^{-C_1 2^{-2k}} \int_{B(\xi, 2^{-k+1}R)} |u(\eta)|\,d\eta.
  \end{equation*}
  Applying the Morrey bound \eqref{ball-est} to each ball $B(\xi, 2^{-k+1}R)$ yields
  \begin{equation*}
        G_{R}(\xi)
         \leq C R^{-Q}  \sum_{k\in \mathbb{Z}}  e^{-C_1 2^{-2k}} (2^{-k}R)^{Q-\alpha}\|u\|_{\mathcal{M}^{1,\alpha}(\mathbb{H}^n)}
         = C  \|u\|_{\mathcal{M}^{1,\alpha}(\mathbb{H}^n)} R^{-\alpha} \left( \sum_{k\in \mathbb{Z}}  e^{-C_1 2^{-2k}} 2^{-k(Q-\alpha)}\right).
  \end{equation*}
  The sum in the parentheses converges: as $k \to \infty$, $2^{-k(Q-\alpha)} \to 0$ because $\alpha<Q$, and as $k \to -\infty$, the super-exponential decay $e^{-C_1 2^{-2k}}$ easily dominates the polynomial growth. This establishes our claim \eqref{inq_G_R<}. Taking the supremum over $R>0$ and $\xi \in \mathbb{H}^n$ yields \eqref{dn-l-3}, completing the proof.
\end{proof}

\begin{proof}[\textbf{Proof of Theorem~\ref{thm_hardy_inter}}]
  First, we consider the endpoint $\theta = p/p^*$. Applying \eqref{dn-l-1} with $q=p^*$, we observe that $\theta = (Q-p)/Q$ and $\theta/(\theta-1) = -(Q-p)/p$. Thus, we have  
  \begin{align*}
        \|u\|_{L^{p^*}(\mathbb{H}^n)}\leq C \|\nabla_{\mathbb{H}^n} u \|_{L^p(\mathbb{H}^n)}^{(Q-p)/Q}\|u\|_{\dot{B}_{\infty,\infty}^{-(Q-p)/p}(\mathbb{H}^n)}^{p/Q}.
    \end{align*}
    Using the Besov--Morrey embedding \eqref{dn-l-3} with $\alpha=\frac{Q-p}{p}$, we get
    \begin{align*}
        \|u\|_{L^{p^*}(\mathbb{H}^n)}\leq C \|\nabla_{\mathbb{H}^n} u \|_{L^p(\mathbb{H}^n)}^{(Q-p)/Q} \|u\|_{\mathcal{M}^{1,\frac{Q-p}{p}}(\mathbb{H}^n)}^{p/Q},
    \end{align*}
   which is the desired inequality for $r=1$ and $\theta=p/p^*$. Finally, for $r\neq 1$, applying \eqref{dn-l-2} with $\gamma=r(Q-p)/p$ and $q=r$, we deduce
     \begin{align*}
        \|u\|_{L^{p^*}(\mathbb{H}^n)}\leq C \|\nabla_{\mathbb{H}^n} u \|_{L^p(\mathbb{H}^n)}^{(Q-p)/Q} \|u\|_{\mathcal{M}^{r,r(Q-p)/p}(\mathbb{H}^n)}^{p/Q}.
    \end{align*}
    This establishes the inequality for $\theta=p/p^*$. 
    
    For the remaining range $\theta \in (p/p^*, 1)$, we recall the standard Folland--Stein embedding (see \cite{VSCC92} for instance):
    \begin{align*}
    S^{1,p}(\mathbb{H}^n) \hookrightarrow L^{p^*}(\mathbb{H}^n).
\end{align*}
By Lemma~\ref{morey-sblv}, we also have
\begin{align*}
    L^{p^*}(\mathbb{H}^n)\hookrightarrow \mathcal{M}^{r, r(Q-p)/p}(\mathbb{H}^n).
\end{align*}
Combining these embeddings with the established endpoint case and interpolating, we obtain the desired inequality for all $p/p^* \leq \theta < 1$, completing the proof.
\end{proof}
\begin{remark}\label{rem_strat}
For our purposes, we state Theorem~\ref{thm_hardy_inter} only in the Heisenberg group setting. Although Theorem~\ref{th-chamar} holds for general stratified Lie groups, so does Lemma~\ref{lem_Mor<Bes}, with suitable modifications. Consequently, Theorem~\ref{thm_hardy_inter} can also be established for general stratified Lie groups. 
\end{remark}

\section{Palais--Smale Profile Decomposition}\label{ps section}
This section is devoted to establishing the Palais-Smale profile decomposition for problem \eqref{eqn_pert_Hn}. For that we recall the Jerison-Lee operator family given by  $\left\{ D_{\eta, \lambda}: \eta \in \mathbb{H}^n, \lambda>0 \right\}$, where $D_{\eta, \lambda}$ is defined in \eqref{eq_jer_Lee_bub}. It is easy to check that for any $(\eta,\lambda)\in \mathbb{H}^n \times (0,\infty)$, $D_{\eta, \lambda}\in \mathcal{U}\left(S^1(\hn)\right)$, the space of all unitary operators on $S^1$, i.e.,
\begin{align}\label{uni-nor}
  \|u\|_{S^1} = \|D_{\eta,\lambda} u\|_{S^1} \qquad \text{ for all } \quad u\in S^1(\hn).
\end{align}
Moreover, if $(\eta_k, \lambda_k)\to (\eta, \lambda)\in \mathbb{H}^n \times (0,\infty)$, for $k\to\infty$, then 
\begin{align}\label{act-con}
 D_{\eta_k, \lambda_k}  \to D_{\eta, \lambda} \text{ in the strong operator norm in }\mathcal{U}\left(S^1(\hn)\right).
\end{align}

The following lemma provides a characterization of the asymptotic orthogonality of two sequences of translation-dilation operators on $S^1(\hn)$.
\begin{lemma}\label{pp-lemma-3}
Let $\{(\xi_k,\lambda_k)\}, \{(\eta_k,r_k)\} \subset \mathbb{H}^n \times (0,\infty)$. 
Then the following statements are equivalent as $k \to \infty$
\begin{enumerate}[(i)]
    \item $\left\langle D_{\xi_k,\lambda_k} u,\, D_{\eta_k,r_k} v \right\rangle_{S^1}
    \to 0, \quad \text{for all } u, v \in S^1(\hn)$,
    
    \item $\left| \log\!\left( \frac{\lambda_k}{r_k} \right) \right| +\left\| \delta_{\lambda_k^{-1}}(\xi_k^{-1}\circ \eta_k) \right\|
     \to \infty.$
\end{enumerate}
\end{lemma}
\begin{proof}
We first observe that, for any $\lambda,\,r>0$ and $\xi,\eta\in\hn$ we have,
\begin{align*}
    \left\langle D_{\xi,\lambda}u, D_{\eta,r}v\right\rangle_{S^1}
    &=\int_{\hn} (\lambda r)^{-\frac{Q}{2}}\nabla_{\hn}u\left(\delta_{\frac{1}{\lambda}}(\xi^{-1}\circ \bar\xi)\right)\cdot\nabla_{\hn}v\left(\delta_{\frac{1}{r}}(\eta^{-1}\circ\bar\xi)\right)\,d\bar \xi\\
    &=\int_{\hn} \left(\frac{r}{\lambda}\right)^{-\frac{Q}{2}}\nabla_{\hn}u(\bar\xi)\cdot\nabla_{\hn}v\left(\delta_{\frac{1}{r}}\left(\eta^{-1}\circ\xi\circ\delta_{\lambda}(\bar\xi)\right)\right)\,d\bar\xi\\
    &=\left\langle u, D_{\delta_{\frac{1}{\lambda}}\left(\xi^{-1}\circ\eta\right),(\frac{r}{\lambda})}v\right\rangle_{S^1},
\end{align*}
holds for all $u,v\in S^1(\hn)$. Therefore, to complete the proof, it is enough to consider $\xi_k=e,\, \lambda_k=1$ for all $k\in\mathbb{N}$. Now, if (ii) does not hold, then up to a subsequence there exists $\eta_k\to \eta\in \mathbb{H}^n$ and $r_k\to r>0$ for $k\to \infty$. Now, start with any $u\not \equiv 0$, then we have
\begin{align*}
    0<\|u\|_{S^1}^2=\lim_{k\to \infty}\langle D_{\eta_k,r_k} u,\, D_{\eta_k,r_k} u \rangle_{S^1}= \lim_{k\to \infty}\langle D_{\eta_k,r_k} u-D_{\eta,r} u,\, D_{\eta_k,r_k} u \rangle_{S^1}+\lim_{k\to \infty}\langle D_{\eta,r} u,\, D_{\eta_k,r_k} u \rangle_{S^1}.
\end{align*}
Thus Cauchy-Schwarz inequality together with \eqref{act-con} yields, 
\begin{align}\label{cs-con}
\left|\langle D_{\eta_k,r_k} u-D_{\eta,r} u,\, D_{\eta_k,r_k} u \rangle_{S^1}\right| \leq \left\|D_{\eta_k,r_k} u-D_{\eta,r} u \right\|_{S^1}\|u \|_{S^1}\rightarrow 0, \text{ as }k\to\infty.
\end{align}

Now, from~$(i)$ with $D_{\xi_k,\lambda_k}u$ replaced by $D_{\eta,r} u$ and $D_{\eta_k,r_k}v$ by $D_{\eta_k,r_k}u$, and then using \eqref{cs-con} and \eqref{uni-nor}, we have,
\begin{align*}
    0= \lim_{k\to \infty}\langle D_{\eta_k,r_k} u,\, D_{\eta_k,r_k} u \rangle_{S^1}=\lim_{k\to\infty}\|D_{\eta_k,r_k}u\|_{S^1}^2=\|u\|_{S^1}^2\neq 0,
\end{align*}
which gives the contradiction. Hence, $(ii)$ holds. 

Conversely, now we assume that condition $(ii)$ holds. Then, up to a subsequence, we can only consider two cases, either $r_k\to 0$ or $r_k\geq r>0$ for $k\to \infty$. Now, thanks to the density property of $C_c^\infty(\hn)$ in $S^1(\hn)$, we can only work with compactly supported smooth functions. Using integration by parts, we have
\begin{align}\label{pp-lem-3-bp}
    \langle u, D_{\eta_k,r_k}v\rangle_{S^1}=\int_{\hn}(\mathcal{L}_{\hn}u)(D_{\eta_k,r_k} v)\, d\xi \qquad \text{ for all }\quad u,\, v\in C_c^\infty(\hn).
\end{align}

\textbf{Case I : $r_k\geq r>0$ for $k\to \infty$.} In this case we have $\|\eta_k\|\to \infty$ and/or $r_k\to \infty$ for $k\to\infty$. Now notice that 
\begin{align*}
     |D_{\eta_k,r_k}v|\leq r^{-\frac{Q-2}{2}}\|v\|_{L^\infty(\hn)}
\end{align*}
and hence $D_{\eta_k,r_k}v$ is uniformly bounded. Also $D_{\eta_k,r_k}v\to 0$ a.e. for $k\to \infty$. Hence, using the dominated convergence theorem in \eqref{pp-lem-3-bp}, we deduce (i).

\textbf{Case II : $r_k\to 0$ for $k\to \infty$.} Let us assume that $v$ is supported in $B(e,R)$, and fix $\varepsilon>0$. Then, for $k$ large, $D_{\eta_k,r_k}v$ is supported in $B(\eta_k,r_kR)\subset B(\eta_k,\varepsilon)$, and therefore, by H\"older's inequality and the invariance of the $L^{2^\star}$-norm, we get 
\[
    \left|\langle u, D_{\eta_k,r_k}v\rangle_{S^1}\right|
    \leq \|\mathcal{L}_{\hn}u\|_{L^\infty(\hn)}\,\|v\|_{L^{2^\star}(\hn)}\,|B(\eta_k,\varepsilon)|^{\frac{Q+2}{2Q}}
    = C\,\|\mathcal{L}_{\hn}u\|_{L^\infty(\hn)}\,\|v\|_{L^{2^\star}(\hn)}\,\varepsilon^{\frac{Q+2}{2}}.
\]
Letting first $k\to\infty$ and then $\varepsilon\to0$, we obtain (i). This completes the proof.
\end{proof}

We next prove a Brezis--Lieb type lemma for the critical nonlinearity, which will be used repeatedly in the sequel.

\begin{lemma}\label{lem_Brezis_Lieb_nonlin}
  Let $\{\varphi_k\}$ be a sequence  converging to $\varphi$ weakly in 
    $S^1(\mathbb{H}^n)$ and pointwise a.e. in $\mathbb{H}^n$.
    If $a \in L^{\infty}(\mathbb{H}^n)$, then
    \[
        a |\varphi_k|^{2^{\star}-2}\varphi_k 
        - a|\varphi_k - \varphi|^{2^{\star}-2}(\varphi_k - \varphi) 
        - a|\varphi|^{2^{\star}-2}\varphi \to 0 
        \quad \text{strongly in } (S^1(\mathbb{H}^n))'.
    \]
\end{lemma}

\begin{proof}
Let us set $\psi_k:=\varphi_k-\varphi$, so that $\psi_k\to0$ a.e.\ and $\sup_k\|\psi_k\|_{L^{2^\star}}\leq M<\infty$. We recall from \cite[Page~182, (11.8)]{AM07} that for every $\varepsilon>0$ there exists $C_\varepsilon>0$ such that $\bigl||s+t|^{2^\star-2}(s+t)-|s|^{2^\star-2}s-|t|^{2^\star-2}t\bigr|^{(2^\star)'}\leq\varepsilon|s|^{2^\star}+C_\varepsilon|t|^{2^\star}$ for all $s,t\in\mathbb{R}$, where $(2^\star)'=\frac{2Q}{Q+2}$. Let $F_k$ denote the $(2^\star)'$-th power of the expression in the statement. Then, taking $s=\psi_k$ and $t=\varphi$, we get $F_k\to0$ a.e.\ and $\bigl(F_k-\varepsilon\|a\|_{L^\infty}^{(2^\star)'}|\psi_k|^{2^\star}\bigr)_+\leq C_\varepsilon\|a\|_{L^\infty}^{(2^\star)'}|\varphi|^{2^\star}$. Hence, by the dominated convergence theorem, $\limsup_k\int_{\mathbb{H}^n}F_k\,d\xi\leq\varepsilon\|a\|_{L^\infty}^{(2^\star)'}M^{2^\star}$. Since $\varepsilon>0$ is arbitrary, the expression tends to zero in $L^{(2^\star)'}(\mathbb{H}^n)$, and hence in $(S^1(\mathbb{H}^n))'$ by the Folland--Stein inequality.
\end{proof}

\begin{remark}\label{rem_ae_convergence}
We note that the pointwise convergence in Lemma~\ref{lem_Brezis_Lieb_nonlin} holds, up to a subsequence, whenever $\varphi_k\rightharpoonup\varphi$ in $S^1(\mathbb{H}^n)$, thanks to the compact embedding $S^1(\mathbb{H}^n)\hookrightarrow L^q_{\mathrm{loc}}(\mathbb{H}^n)$, $1\leq q<2^\star$ \cite{FS74}, and a diagonal argument.
\end{remark}

\begin{proof}[\textbf{Proof of Theorem \ref{thm_PS_decomp}}]
Let $\{u_k\} \subset S^1(\mathbb{H}^n)$ be a $(PS)_c$ sequence for 
$\mathcal{E}_{a,f}$, that is as $k\to\infty$, we have $\mathcal{E}_{a,f}(u_k) \to c$ and 
$\mathcal{E}_{a,f}'(u_k) \to 0$ in $(S^1(\mathbb{H}^n))'$. Explicitly, for 
every $\varphi \in S^1(\mathbb{H}^n)$,
\begin{equation}\label{exp_E,a,f_prime}
   {}_{(S^1)'}\left\langle\mathcal{E}_{a,f}'(u_k), \varphi \right\rangle_{S^1}=\int_{\mathbb{H}^n} \nabla_{\mathbb{H}^n} u_k \cdot \nabla_{\mathbb{H}^n} 
    \varphi \, d\xi 
    - \int_{\mathbb{H}^n} a(\xi)|u_k|^{2^{\star}-2}u_k\varphi\, d\xi 
    - {}_{(S^1)'}\langle f, \varphi\rangle_{S^1} 
    = o(1).
\end{equation}

\smallskip\noindent
\textbf{Step I: Boundedness of  $(PS)_c$ sequence.}

Testing \eqref{exp_E,a,f_prime} with $\varphi = u_k$ and combining the result with the expression for the energy $\mathcal{E}_{a,f}(u_k)$, the critical nonlinear terms cancel exactly, yielding
\begin{align*}
    \mathcal{E}_{a,f}(u_k) - \frac{1}{2^{\star}}{}_{(S^1)'}\langle\mathcal{E}_{a,f}'(u_k), u_k\rangle_{S^1} 
   & = \left(\frac{1}{2} - \frac{1}{2^{\star}}\right)\|u_k\|_{S^1}^2 
    - \left(1 - \frac{1}{2^{\star}}\right){}_{(S^1)'}\langle f, u_k\rangle_{S^1}\\&\geq \frac{1}{Q}\|u_k\|_{S^1}^2 - \left(1 - \frac{1}{2^{\star}}\right) \|f\|_{(S^1)'}\|u_k\|_{S^1}.
\end{align*}
Since $\mathcal{E}_{a,f}(u_k) \to c$ and $\left\|\mathcal{E}_{a,f}'(u_k)\right\|_{(S^1)'} = o(1)$ as $k \to \infty$, it follows from a standard argument that the sequence $\{u_k\}$ must be uniformly bounded in $S^1(\mathbb{H}^n)$. Furthermore, by the Banach-Alaoglu theorem and Remark \ref{rem_ae_convergence}, we extract 
a subsequence (still denoted $\{u_k\}$) and $u^{(0)} \in S^1(\mathbb{H}^n)$ 
such that $u_k \rightharpoonup u^{(0)}$ weakly in $S^1(\mathbb{H}^n)$ and 
$u_k \to u^{(0)}$ pointwise a.e.\ in $\mathbb{H}^n$.

\smallskip\noindent
\textbf{Step II: The weak limit $u^{(0)}$ is a weak solution to 
\eqref{eqn_pert_Hn}.}

\smallskip
Set $\sigma_k := u_k - u^{(0)}$, so $\sigma_k\rightharpoonup 0$ weakly in 
$S^1(\mathbb{H}^n)$ and \textup{(up to a subsequence, still denoted by $\{\sigma_k\}$)} $\sigma_k\to 0$ a.e.\ by Remark~\ref{rem_ae_convergence}. 
Fix $\varphi\in C_c^\infty(\mathbb{H}^n)$ with $K:=\operatorname{supp}(\varphi)$, 
and pass to the limit in \eqref{exp_E,a,f_prime} term by term. By weak convergence in $S^1(\mathbb{H}^n)$,
\begin{align}\label{eq_grad_II}
    \int_{\mathbb{H}^n}\nabla_{\mathbb{H}^n}u_k\cdot
    \nabla_{\mathbb{H}^n}\varphi\,d\xi
    \longrightarrow
    \int_{\mathbb{H}^n}\nabla_{\mathbb{H}^n}u^{(0)}\cdot
    \nabla_{\mathbb{H}^n}\varphi\,d\xi.
\end{align}
Applying Lemma~\ref{lem_Brezis_Lieb_nonlin} to $(u_k)$ with weak limit $u^{(0)}$ gives us
\begin{align}\label{eq_nonlin_II}
    a|u_k|^{2^{\star}-2}u_k
    = a|\sigma_k|^{2^{\star}-2}\sigma_k + a|u^{(0)}|^{2^{\star}-2}u^{(0)} + \varepsilon_k,
\end{align}
where $\varepsilon_k\to 0$ strongly in $(S^1(\mathbb{H}^n))'$, so 
$|{}_{(S^1)'}\langle\varepsilon_k,\varphi\rangle_{S^1}|\leq\|\varepsilon_k\|_{(S^1)'}\|\varphi\|_{S^1}\to 0$. To show that the remainder term involving $\sigma_k$ vanishes, let $K = \operatorname{supp}(\varphi)$. Since $\varphi \in C_c^\infty(\mathbb{H}^n)$, we can bound the integral using the $L^\infty$-norm of the test function:
\begin{align*}
    \left|\int_K a |\sigma_k|^{2^{\star}-2}\sigma_k\,\varphi\,d\xi\right| 
    \leq \|a\|_{L^{\infty}(\hn)} \|\varphi\|_{L^{\infty}(K)} \int_K |\sigma_k|^{2^{\star}-1}\,d\xi \leq  \|a\|_{L^{\infty}(\hn)} \|\varphi\|_{L^{\infty}(K)} \|\sigma_k\|_{L^{2^{\star}-1}(K)}^{2^{\star}-1}.
\end{align*}
Because $2^{\star}-1 < 2^{\star}$, the local embedding $S^1(\mathbb{H}^n)\hookrightarrow L^{2^{\star}-1}_{\mathrm{loc}}(\mathbb{H}^n)$ is compact (see Remark~\ref{rem_ae_convergence}). Since $\sigma_k \rightharpoonup 0$ weakly in $S^1(\mathbb{H}^n)$, it converges strongly to zero in $L^{2^{\star}-1}(K)$. Therefore, the right-hand side of the inequality vanishes as $k \to \infty$, yielding
\begin{align}\label{eq_last_ter_II}
    \int_{\mathbb{H}^n}a(\xi)|u_k|^{2^{\star}-2}u_k\,\varphi\,d\xi
    \longrightarrow
    \int_{\mathbb{H}^n}a(\xi)|u^{(0)}|^{2^{\star}-2}u^{(0)}\,\varphi\,d\xi.
\end{align}
Utilizing \eqref{eq_grad_II}, \eqref{eq_nonlin_II}, and \eqref{eq_last_ter_II},  we can pass the limit in \eqref{exp_E,a,f_prime}, whence we obtain
\[
    \int_{\mathbb{H}^n}\nabla_{\mathbb{H}^n}u^{(0)}\cdot
    \nabla_{\mathbb{H}^n}\varphi\,d\xi
    -\int_{\mathbb{H}^n}a(\xi)|u^{(0)}|^{2^{\star}-2}u^{(0)}\,\varphi\,d\xi
    -{}_{(S^1)'}\langle f,\varphi\rangle_{S^1}=0
\]
for all $\varphi\in C_c^\infty(\mathbb{H}^n)$. Since $C_c^\infty(\mathbb{H}^n)$ 
is dense in $S^1(\mathbb{H}^n)$, $u^{(0)}$ is a weak solution to 
\eqref{eqn_pert_Hn}.

\smallskip\noindent
\textbf{Step III: $\{\sigma_k\}$ is a $(PS)$ sequence for $\mathcal{E}_{a,0}$ 
at level $c-\mathcal{E}_{a,f}(u^{(0)})$.}

\smallskip
We recall that $\sigma_k=u_k-u^{(0)}\rightharpoonup 0$ weakly in $S^1(\mathbb{H}^n)$, which gives 
$\langle u_k,u^{(0)}\rangle_{S^1}\to\|u^{(0)}\|_{S^1}^2$. Hence 
\begin{align}\label{eq_ener_III}
    \|\sigma_k\|_{S^1}^2
    =\|u_k\|_{S^1}^2-2\langle u_k,u^{(0)}\rangle_{S^1}
    +\|u^{(0)}\|_{S^1}^2
    =\|u_k\|_{S^1}^2-\|u^{(0)}\|_{S^1}^2+o(1).
\end{align}
By the classical Br\'{e}zis--Lieb lemma with weight 
$a\in L^\infty(\mathbb{H}^n)$,
\begin{align}\label{eq_app_BL_III}
    \int_{\mathbb{H}^n}a(\xi)|\sigma_k|^{2^{\star}}\,d\xi
    =\int_{\mathbb{H}^n}a(\xi)|u_k|^{2^{\star}}\,d\xi
    -\int_{\mathbb{H}^n}a(\xi)|u^{(0)}|^{2^{\star}}\,d\xi+o(1).
\end{align}
Since $\sigma_k\rightharpoonup 0$ in $S^1(\mathbb{H}^n)$ and $f\in(S^1(\mathbb{H}^n))'$, 
we have ${}_{(S^1)'}\langle f,\sigma_k\rangle_{S^1}\to 0$. Combining \eqref{eq_ener_III} and  \eqref{eq_app_BL_III}
\begin{align*}
    \mathcal{E}_{a,0}(\sigma_k)
    &=\frac{1}{2}\|\sigma_k\|_{S^1}^2
    -\frac{1}{2^{\star}}\int_{\mathbb{H}^n}a(\xi)|\sigma_k|^{2^{\star}}\,d\xi\\
    &=\Bigl(\frac{1}{2}\|u_k\|_{S^1}^2
    -\frac{1}{2^{\star}}\int_{\mathbb{H}^n}a(\xi)|u_k|^{2^{\star}}\,d\xi
    -{}_{(S^1)'}\langle f,u_k\rangle_{S^1}\Bigr)\\
    &\quad-\Bigl(\frac{1}{2}\|u^{(0)}\|_{S^1}^2
    -\frac{1}{2^{\star}}\int_{\mathbb{H}^n}a(\xi)|u^{(0)}|^{2^{\star}}\,d\xi
    -{}_{(S^1)'}\langle f,u^{(0)}\rangle_{S^1}\Bigr)+o(1)\\
    &=\mathcal{E}_{a,f}(u_k)-\mathcal{E}_{a,f}(u^{(0)})+o(1)
    \longrightarrow c-\mathcal{E}_{a,f}(u^{(0)}).
\end{align*}
To show that the derivative vanishes, let $\varphi\in S^1(\mathbb{H}^n)$. Since $u^{(0)}$ solves \eqref{eqn_pert_Hn} (Step~II), we have ${}_{(S^1)'}\left\langle\mathcal{E}_{a,f}'(u^{(0)}),\varphi\right\rangle_{S^1}=0$. Therefore, we can expand the derivative as follows:
\begin{align*}
    {}_{(S^1)'}\left\langle \mathcal{E}_{a,0}'(\sigma_k),\varphi\right\rangle_{S^1}
    &=\langle u_k,\varphi\rangle_{S^1}
    -\int_{\mathbb{H}^n}a(\xi)|u_k|^{2^{\star}-2}u_k\,\varphi\,d\xi
    -{}_{(S^1)'}\langle f,\varphi\rangle_{S^1}\\
    &\quad-\langle u^{(0)},\varphi\rangle_{S^1}
    +\int_{\mathbb{H}^n}a(\xi)|u^{(0)}|^{2^{\star}-2}u^{(0)}\,\varphi\,d\xi
    +{}_{(S^1)'}\langle f,\varphi\rangle_{S^1}\\
    &\quad+\int_{\mathbb{H}^n}a(\xi)\Bigl\{|u_k|^{2^{\star}-2}u_k
    -|u^{(0)}|^{2^{\star}-2}u^{(0)}
    -|\sigma_k|^{2^{\star}-2}\sigma_k\Bigr\}\varphi\,d\xi\\
    &=o(1)
    +\int_{\mathbb{H}^n}a(\xi)\Bigl\{|u_k|^{2^{\star}-2}u_k
    -|u^{(0)}|^{2^{\star}-2}u^{(0)}
    -|\sigma_k|^{2^{\star}-2}\sigma_k\Bigr\}\varphi\,d\xi,
\end{align*}
where we used the fact that ${}_{(S^1)'}\left\langle\mathcal{E}_{a,f}'(u_k),\varphi\right\rangle_{S^1} =o(1)$ by \eqref{exp_E,a,f_prime}, and ${}_{(S^1)'}\left\langle\mathcal{E}_{a,f}'(u^{(0)}),\varphi\right\rangle_{S^1}=0$. By Lemma~\ref{lem_Brezis_Lieb_nonlin}, it directly follows that for all $\varphi\in S^1(\mathbb{H}^n)$,
\[
    \int_{\mathbb{H}^n}a(\xi)\Bigl\{|u_k|^{2^{\star}-2}u_k
    -|u^{(0)}|^{2^{\star}-2}u^{(0)}
    -|\sigma_k|^{2^{\star}-2}\sigma_k\Bigr\}\varphi\,d\xi=o(1).
\]
Consequently, we obtain $\left\|\mathcal{E}_{a,0}'(\sigma_k)\right\|_{(S^1)'} \to 0$, for $k\to \infty$, which completes Step~III.  

\smallskip\noindent
\textbf{Step IV: Concentration and rescaling.}

\smallskip
From Step~III, $\{\sigma_k\}$ is a $(PS)$ sequence for $\mathcal{E}_{a,0}$ 
with $\sigma_k\rightharpoonup 0$ weakly in $S^1(\mathbb{H}^n)$. If 
$\sigma_k\to 0$ strongly in $S^1(\mathbb{H}^n)$, the theorem holds with 
$m=0$ and we are done. Henceforth we assume $\sigma_k\not\to 0$ strongly in 
$S^1(\mathbb{H}^n)$.

Testing ${}_{(S^1)'}\left\langle\mathcal{E}_{a,0}'(\sigma_k),\sigma_k\right\rangle_{S^1}
=o(1)$ against $\sigma_k$ gives
\[
    \|\sigma_k\|_{S^1}^2
    =\int_{\mathbb{H}^n}a(\xi)|\sigma_k|^{2^{\star}}\,d\xi+o(1).
\]
If $\sigma_k\to 0$ in $L^{2^{\star}}(\mathbb{H}^n)$, the right-hand side vanishes 
and forces $\|\sigma_k\|_{S^1}\to 0$, contradicting the assumption. Hence, 
up to a subsequence,
\begin{equation}\label{eq:infimum_bound}
    \inf_k\|\sigma_k\|_{L^{2^{\star}}(\mathbb{H}^n)}\geq C_0>0.
\end{equation}

Moreover, by Lemma~\ref{morey-sblv} with $p=2$, the continuous embedding $L^{2^{\star}}(\mathbb{H}^n)\hookrightarrow\mathcal{M}^{r,r(Q-2)/2}(\mathbb{H}^n)$ holds for any fixed $1\leq r<2^{\star}$. Combined with the uniform $S^1$-bound of $\{\sigma_k\}$, this yields 
\begin{equation}\label{eq:morrey_upper}
    \sup_k\|\sigma_k\|_{\mathcal{M}^{r,r(Q-2)/2}(\mathbb{H}^n)}\leq C_1<\infty.
\end{equation}
For the lower bound, applying the improved Sobolev interpolation inequality Theorem~\ref{thm_hardy_inter} with $p=2$ and 
$\theta\in[2/2^{\star},1)$ gives
\[
    \|\sigma_k\|_{L^{2^{\star}}}
    \leq C\|\sigma_k\|_{S^1}^{\theta}
    \|\sigma_k\|_{\mathcal{M}^{r,r(Q-2)/2}}^{1-\theta}.
\]
Since $\|\sigma_k\|_{S^1}\leq M$ and $\|\sigma_k\|_{L^{2^{\star}}}\geq C_0>0$ by 
\eqref{eq:infimum_bound}, rearranging yields
\begin{equation}\label{eq:morrey_lower}
    \inf_k\|\sigma_k\|_{\mathcal{M}^{r,r(Q-2)/2}(\mathbb{H}^n)}\geq C_2>0.
\end{equation}
Together, \eqref{eq:morrey_upper} and \eqref{eq:morrey_lower} give 
$0<C_2\leq\|\sigma_k\|_{\mathcal{M}^{r,r(Q-2)/2}}\leq C_1<\infty$ for all 
$k$. By Definition~\ref{def_mor}, for each $k\in\mathbb{N}$ there 
exist $\lambda_k^{(1)}>0$ and $\xi_k^{(1)}\in\mathbb{H}^n$ such that  
\begin{equation}\label{eq:concentration}
    \frac{(\lambda_k^{(1)})^{r(Q-2)/2}}{|B(\xi_k^{(1)},\lambda_k^{(1)})|}
    \int_{B(\xi_k^{(1)},\lambda_k^{(1)})}|\sigma_k(\xi)|^r\,d\xi
    \geq\frac{C_2^r}{2}=C_3>0,
\end{equation}
uniformly in $k$.

We claim that $|\log\lambda_k^{(1)}|+\|\xi_k^{(1)}\|\to\infty$ up to a 
subsequence. Suppose for contradiction that the sequence $|\log\lambda_k^{(1)}|+\|\xi_k^{(1)}\|$  does not diverge to infinity. Then, passing to a subsequence, there exist positive constants  $0 < c \leq C < \infty$ such that $\lambda_k^{(1)} \in [c, C]$ and $\|\xi_k^{(1)}\| \leq C$  for all $k$. So any $\xi \in B(\xi_k^{(1)}, \lambda_k^{(1)})$  satisfies $\|\xi\| \leq \|\xi_k^{(1)}\| + \lambda_k^{(1)} \leq 2C$. Choosing a fixed radius $R:= 2C+1$, it follows that the domains of concentration are uniformly  contained in a single large ball
\[
    B(\xi_k^{(1)}, \lambda_k^{(1)}) \subset B(e, R) \quad \text{for all } k.
\]
Using the concentration assumption \eqref{eq:concentration} and the fact that the scales are strictly bounded away from zero ($\lambda_k^{(1)} \geq c$), we bound the integral from below
\begin{align}\label{eq_Loc_Lr_sig>}
    \int_{B(e, R)}|\sigma_k(\xi)|^r\,d\xi 
    \geq \int_{B(\xi_k^{(1)}, \lambda_k^{(1)})}|\sigma_k(\xi)|^r\,d\xi 
    \geq{ C_3 \frac{|B(\xi_k^{(1)}, \lambda_k^{(1)})|}{\left(\lambda_k^{(1)}\right)^{r\frac{(Q-2)}{2}}}}  
    &={ C_3 (\lambda_k^{(1)})^{Q-r\frac{(Q-2)}{2}} |B(e,1)|} =: \tilde{C} > 0.
\end{align}
However, the sequence $\{\sigma_k\}$ satisfies $\sigma_k \rightharpoonup 0$ weakly in  $S^1(\mathbb{H}^n)$.  The 
subelliptic Rellich--Kondrachov theorem then gives ${\sigma}_k\to 0$ strongly in $L^r_{\mathrm{loc}}(\mathbb{H}^n)$, 
contradicting \eqref{eq_Loc_Lr_sig>}. Therefore, our assumption that the sequence of centers and log-scales remains bounded must be false, establishing the claim that
\begin{align}\label{claim_log+||}
    |\log\lambda_k^{(1)}|+\|\xi_k^{(1)}\|\to\infty \qquad \text{ as } k\to \infty.
\end{align}
Motivated by the Jerison--Lee operator 
$D_{\eta,\lambda}u(\xi)=\lambda^{-\frac{Q-2}{2}}u\left(\delta_{\frac{1}{\lambda}}(\eta^{-1}\circ\xi)\right)$ from Theorem~\ref{thm_JL_bubble}, we define the rescaled sequence
\begin{equation}\label{eq:rescaled_unified}
    \tilde{\sigma}_k(\eta)
    :=(\lambda_k^{(1)})^{\frac{Q-2}{2}}
    \sigma_k\!\left(\xi_k^{(1)}\circ\delta_{\lambda_k^{(1)}}(\eta)\right).
\end{equation}
By the scale invariance of the $S^1$-norm, we have
\[
    \|\tilde{\sigma}_k\|_{S^1}
    =\Bigl\|(\lambda_k^{(1)})^{\frac{Q-2}{2}}
    \sigma_k(\xi_k^{(1)}\circ\delta_{\lambda_k^{(1)}}(\,\cdot\,))
    \Bigr\|_{S^1}
    =\|\sigma_k\|_{S^1}\leq M\,\text{ for all }k\in\mathbb{N}.
\]
Since $\{\tilde{\sigma}_k\}$ is bounded in $S^1(\mathbb{H}^n)$, the Banach--Alaoglu theorem and the subelliptic Rellich--Kondr\\-achov theorem (Remark~\ref{rem_ae_convergence}) guarantee that, up to a subsequence, there exists $u^{(1)}\in S^1(\mathbb{H}^n)$ such that
\[
    \tilde{\sigma}_k\rightharpoonup u^{(1)}
    \quad\text{weakly in }S^1(\mathbb{H}^n),
    \qquad
    \tilde{\sigma}_k\to u^{(1)}
    \quad\text{in }L^r_{\mathrm{loc}}(\mathbb{H}^n)\quad(1\leq r<2^{\star}).
\]
Consequently, the extracted bubble in the decomposition \eqref{eq_seq_decomp} takes the form
\[
    (\lambda_k^{(1)})^{-\frac{Q-2}{2}}
    u^{(1)}\!\left(\delta_{1/\lambda_k^{(1)}}
    \!\left((\xi_k^{(1)})^{-1}\circ\xi\right)\right)
    = D_{\xi_k^{(1)},\lambda_k^{(1)}}u^{(1)}(\xi).
\]
Moreover, \eqref{eq:concentration} gives us 
\[
    0<C_3
    \leq\frac{1}{|B(e,1)|}
    \int_{B(e,1)}|\tilde{\sigma}_k(\eta)|^r\,d\eta
    \longrightarrow
    \frac{1}{|B(e,1)|}
    \int_{B(e,1)}|u^{(1)}(\eta)|^r\,d\eta,
\]
so $u^{(1)}\not\equiv 0$.  

\smallskip\noindent
\textbf{Step V: The rescaled limit $u^{(1)}$ solves the limiting equation
\eqref{eq:limiting_eq_j}.}

\smallskip
We recall from Step~IV the rescaled sequence \eqref{eq:rescaled_unified}, 
where $|\log\lambda_k^{(1)}| +\|\xi_{k}^{(1)}\|  \rightarrow \infty$,
$u^{(1)}\not\equiv 0$ established there. We show that $u^{(1)}$ is a nontrivial weak solution to \eqref{eq:limiting_eq_j}.
Let us fix $\varphi\in C_c^\infty(\mathbb{H}^n)$. Since 
$\tilde{\sigma}_k\rightharpoonup u^{(1)}$ weakly in $S^1(\mathbb{H}^n)$,
\[
    \langle u^{(1)},\varphi\rangle_{S^1}
    =\lim_{k\to\infty}\langle\tilde{\sigma}_k,\varphi\rangle_{S^1}.
\]
We define the rescaled test function
\begin{equation}\label{eq:rescaled_test}
    \varphi_k(\xi)
    :=D_{\xi_k^{(1)},\,\lambda_k^{(1)}}\varphi(\xi)
    =(\lambda_k^{(1)})^{-\frac{Q-2}{2}}
    \varphi\!\left(\delta_{1/\lambda_k^{(1)}}
    \!\left((\xi_k^{(1)})^{-1}\circ\xi\right)\right),
\end{equation}
which satisfies $\|\varphi_k\|_{S^1}=\|\varphi\|_{S^1}$. We claim $\varphi_k\rightharpoonup 0$ weakly in 
$S^1(\mathbb{H}^n)$: indeed, by Lemma~\ref{pp-lemma-3}, 
$D_{\xi_k^{(1)},1/\lambda_k^{(1)}}v\rightharpoonup 0$ for all 
$v\in S^1(\mathbb{H}^n)$ if and only if 
$|\log(1/\lambda_k^{(1)})|+\|\xi_k^{(1)}\|
=|\log\lambda_k^{(1)}|+\|\xi_k^{(1)}\|\to\infty$, which holds by 
\eqref{claim_log+||}.

By the $S^1(\hn)$ inner product invariance and the 
change of variables,
\begin{equation}\label{eq:inner_prod_identity}
    \langle\tilde{\sigma}_k,\varphi\rangle_{S^1}
    =\langle\sigma_k,\varphi_k\rangle_{S^1}.
\end{equation}
Since ${}_{(S^1)'}\left\langle\mathcal{E}_{a,0}'(\sigma_k),\varphi_k\right\rangle_{S^1}
=o(1)$ from Step~III, we have
\[
    \langle\sigma_k,\varphi_k\rangle_{S^1}
    =\int_{\mathbb{H}^n}a(\xi)|\sigma_k(\xi)|^{2^{\star}-2}\sigma_k(\xi)\,
    \varphi_k(\xi)\,d\xi+o(1).
\]
Applying the change of variables 
$\xi=\xi_k^{(1)}\circ\delta_{\lambda_k^{(1)}}(\eta)$ with Jacobian 
$(\lambda_k^{(1)})^Q$, and substituting 
the definitions of $\tilde{\sigma}_k$ and $\varphi_k$,
gives us
\begin{equation}\label{eq:step5_main}
    \langle\tilde{\sigma}_k,\varphi\rangle_{S^1}
    =\int_{\mathbb{H}^n}
    a_k(\eta)
    |\tilde{\sigma}_k(\eta)|^{2^{\star}-2}\tilde{\sigma}_k(\eta)\,
    \varphi(\eta)\,d\eta+o(1),
\end{equation}
where $a_k(\eta):=a(\xi_k^{(1)}\circ\delta_{\lambda_k^{(1)}}(\eta))$.
To pass to the limit, we estimate
\begin{align*}
    &\left|\int_{\mathbb{H}^n}
    a_k(\eta)|\tilde{\sigma}_k|^{2^{\star}-2}\tilde{\sigma}_k\,\varphi\,d\eta
    -a_\infty^{(1)}\int_{\mathbb{H}^n}
    |u^{(1)}|^{2^{\star}-2}u^{(1)}\,\varphi\,d\eta\right|
    \leq I_k+J_k,
\end{align*}
where
\begin{align*}
    I_k:=\left|\int_{\mathbb{H}^n}
    a_k(\eta)\!\left(|\tilde{\sigma}_k|^{2^{\star}-2}\tilde{\sigma}_k
    -|u^{(1)}|^{2^{\star}-2}u^{(1)}\right)\varphi\,d\eta\right|, \text{ and }
    J_k:=\left|\int_{\mathbb{H}^n}
    \!\left(a_k(\eta)-a_\infty^{(1)}\right)
    |u^{(1)}|^{2^{\star}-2}u^{(1)}\,\varphi\,d\eta\right|.
\end{align*}
For $J_k$, we note that $a_{k}(\eta) \to a_{\infty}^{(1)}$ for a.e.\ $\eta$, after passing to a subsequence. Indeed, by \eqref{eq:param_behavior}, either $\lambda_k^{(1)}\to0$ and $\xi_k^{(1)}\to\xi^{(1)}$, in which case $a_k\to a(\xi^{(1)})$ pointwise by the continuity of $a$, or else, for every $M>0$ and $R>0$, the set of $\eta\in B(e,R)$ with $\|\xi_k^{(1)}\circ\delta_{\lambda_k^{(1)}}(\eta)\|\leq M$ has measure tending to zero, which in turn implies that $a_k\to1$ in measure on $B(e,R)$ by Assumption~\hyperref[assum_main]{\textbf{(A)}}. Therefore, applying the dominated convergence theorem, we conclude that $J_k \to 0$ as $k \to \infty$. 

For $I_k$, since $\tilde{\sigma}_k\to u^{(1)}$ a.e.\ in $\mathbb{H}^n$ 
(Remark~\ref{rem_ae_convergence}), $\{\tilde{\sigma}_k\}$ is uniformly bounded 
in $L^{2^{\star}}(\mathbb{H}^n)$ and the Folland–Stein-Sobolev embedding, the compact support of $\varphi$, and the Vitali convergence theorem 
give $I_k\to 0$. Combining these observations, \eqref{eq:step5_main} yields
\[
    \langle u^{(1)},\varphi\rangle_{S^1}
    =a_\infty^{(1)}\int_{\mathbb{H}^n}
    |u^{(1)}|^{2^{\star}-2}u^{(1)}\,\varphi\,d\eta
\]
for all $\varphi\in C_c^\infty(\mathbb{H}^n)$. Since $C_c^\infty(\mathbb{H}^n)$ 
is dense in $S^1(\mathbb{H}^n)$, this extends to all $\varphi\in S^1(\mathbb{H}^n)$, 
so $u^{(1)}$ is a nontrivial weak solution to \eqref{eq:limiting_eq_j} with 
$j=1$, completing Step~V.

\smallskip\noindent
\textbf{Step VI: $\{v_k\}$ is a $(PS)$ sequence for $\mathcal{E}_{a,0}$ at 
a strictly lower energy level.}

\smallskip
We recall the rescaled sequence $\{\tilde{\sigma}_k\}$ defined in \eqref{eq:rescaled_unified} and the nontrivial weak solution $u^{(1)}$ to \eqref{eq:limiting_eq_j} extracted in Step~V. We define the remainder sequence by
\begin{equation*}
    v_k(\xi) 
    := \sigma_k(\xi) - D_{\xi_k^{(1)},\lambda_k^{(1)}}u^{(1)}(\xi)
    =\sigma_k(\xi) 
    - (\lambda_k^{(1)})^{-\frac{Q-2}{2}}
    u^{(1)}\!\left(\delta_{1/\lambda_k^{(1)}}
    \!\left((\xi_k^{(1)})^{-1}\circ\xi\right)\right).
\end{equation*}
The corresponding rescaled remainder is given by
\begin{equation}\label{eq:vk_rescaled}
    \hat{v}_k(\eta) 
    := (\lambda_k^{(1)})^{\frac{Q-2}{2}}
    v_k\!\left(\xi_k^{(1)}\circ\delta_{\lambda_k^{(1)}}(\eta)\right)
    = \tilde{\sigma}_k(\eta) - u^{(1)}(\eta),
\end{equation}
and by the scale invariance of the $S^1$-norm, we obtain
\begin{equation}\label{eq:vk_norm}
    \|v_k\|_{S^1(\mathbb{H}^n)} 
    = \|\hat{v}_k\|_{S^1(\mathbb{H}^n)} 
    = \|\tilde{\sigma}_k - u^{(1)}\|_{S^1(\mathbb{H}^n)}.
\end{equation}
Furthermore, we note that $v_k\rightharpoonup 0$ weakly in $S^1(\mathbb{H}^n)$. Indeed, $\sigma_k\rightharpoonup 0$ weakly by Step~III, and the parameter condition \eqref{claim_log+||} combined with Lemma~\ref{pp-lemma-3} ensures that the extracted bubble satisfies $D_{\xi_k^{(1)},\lambda_k^{(1)}}u^{(1)}\rightharpoonup 0$ weakly in $S^1(\mathbb{H}^n)$. 

\smallskip
{\noindent\textbf{Brezis--Lieb for the rescaled nonlinearity.} Recall $a_k(\eta)=a(\xi_k^{(1)}\circ\delta_{\lambda_k^{(1)}}(\eta))$ and, as shown in Step~V, $a_k(\eta)\to a_\infty^{(1)}$ for a.e.\ $\eta$. Now using Brezis–Lieb lemma, we have
\begin{equation}\label{eq:BL_rescaled_1}
    \int_{\mathbb{H}^n} a_k(\eta)|\tilde{\sigma}_k(\eta)-u^{(1)}(\eta)|^{2^{\star}}
    \,d\eta
    = \int_{\mathbb{H}^n} a_k(\eta)|\tilde{\sigma}_k(\eta)|^{2^{\star}}\,d\eta
    - \int_{\mathbb{H}^n} a_{k}(\eta)|u^{(1)}(\eta)|^{2^{\star}}\,d\eta
    + o(1).
\end{equation}
Also by the the dominated convergence theorem we obtain
\begin{equation}\label{eq:BL_rescaled_2}
\int_{\mathbb{H}^n}a_k(\eta)|u^{(1)}(\eta)|^{2^{\star}}\,d\eta
    \to\int_{\mathbb{H}^n}a_\infty^{(1)}|u^{(1)}(\eta)|^{2^{\star}}\,d\eta.    
\end{equation}
Combining \eqref{eq:BL_rescaled_1} and \eqref{eq:BL_rescaled_2}, we deduce
\begin{equation}\label{eq:BL_rescaled}
    \int_{\mathbb{H}^n} a_k(\eta)|\tilde{\sigma}_k(\eta)-u^{(1)}(\eta)|^{2^{\star}}
    \,d\eta
    = \int_{\mathbb{H}^n} a_k(\eta)|\tilde{\sigma}_k(\eta)|^{2^{\star}}\,d\eta
    - \int_{\mathbb{H}^n} a_\infty^{(1)}|u^{(1)}(\eta)|^{2^{\star}}\,d\eta
    + o(1).
\end{equation}}

\smallskip
\noindent\textbf{Energy identity for $\{v_k\}$.}
We compute $\mathcal{E}_{a,0}(v_k)$ by passing through the rescaled 
variables. By \eqref{eq:vk_norm} and the change of variables 
$\xi=\xi_k^{(1)}\circ\delta_{\lambda_k^{(1)}}(\eta)$,  
gives us
\begin{align*}
    \mathcal{E}_{a,0}(v_k)
    &= \frac{1}{2}\|v_k\|_{S^1}^2
    - \frac{1}{2^{\star}}\int_{\mathbb{H}^n}a(\xi)|v_k(\xi)|^{2^{\star}}\,d\xi \\
    &= \frac{1}{2}\|\tilde{\sigma}_k-u^{(1)}\|_{S^1}^2
    - \frac{1}{2^{\star}}\int_{\mathbb{H}^n}
    a_k(\eta)|\tilde{\sigma}_k(\eta)-u^{(1)}(\eta)|^{2^{\star}}\,d\eta.
\end{align*}
For the norm, expanding and using $\tilde{\sigma}_k\rightharpoonup u^{(1)}$ 
weakly in $S^1(\mathbb{H}^n)$ gives 
$\langle\tilde{\sigma}_k,u^{(1)}\rangle_{S^1}\to\|u^{(1)}\|_{S^1}^2$, 
so $\|\tilde{\sigma}_k-u^{(1)}\|_{S^1}^2
=\|\tilde{\sigma}_k\|_{S^1}^2-\|u^{(1)}\|_{S^1}^2+o(1)$.
For the integral, we apply  \eqref{eq:BL_rescaled}. Combining,
\begin{align*}
    \mathcal{E}_{a,0}(v_k)
    &= \frac{1}{2}\!\left(\|\tilde{\sigma}_k\|_{S^1}^2
    - \|u^{(1)}\|_{S^1}^2\right)
    - \frac{1}{2^{\star}}\int_{\mathbb{H}^n}
    a_k|\tilde{\sigma}_k|^{2^{\star}}\,d\eta
    + \frac{1}{2^{\star}}\int_{\mathbb{H}^n}
    a_\infty^{(1)}|u^{(1)}|^{2^{\star}}\,d\eta
    + o(1).
\end{align*}
Using
$\|\tilde{\sigma}_k\|_{S^1}=\|\sigma_k\|_{S^1}$, and reversing the change 
of variables converts 
$\int a_k|\tilde{\sigma}_k|^{2^{\star}}\,d\eta
=\int a(\xi)|\sigma_k|^{2^{\star}}\,d\xi$. Therefore,
\begin{align*}
    \mathcal{E}_{a,0}(v_k)
    &= \frac{1}{2}\|\sigma_k\|_{S^1}^2
    - \frac{1}{2^{\star}}\int_{\mathbb{H}^n}a(\xi)|\sigma_k|^{2^{\star}}\,d\xi
    - \left(\frac{1}{2}\|u^{(1)}\|_{S^1}^2
    - \frac{1}{2^{\star}}\int_{\mathbb{H}^n}
    a_\infty^{(1)}|u^{(1)}|^{2^{\star}}\,d\eta\right)
    + o(1).
\end{align*}

We observe that the term in parentheses is exactly $\mathcal{E}_{a_\infty^{(1)},0}(u^{(1)})$, and therefore
\begin{equation}\label{eq:energy_vk}
    \mathcal{E}_{a,0}(v_k)
    = \mathcal{E}_{a,0}(\sigma_k)
    - \mathcal{E}_{a_\infty^{(1)},0}(u^{(1)})
    + o(1)
    \longrightarrow
    c - \mathcal{E}_{a,f}(u^{(0)})
    - \mathcal{E}_{a_\infty^{(1)},0}(u^{(1)}),
\end{equation}
where we used the fact $\mathcal{E}_{a,0}(\sigma_k)\to c-\mathcal{E}_{a,f}(u^{(0)})$ 
from Step~III. Moreover, a direct rescaling shows that
$\mathcal{E}_{a_\infty^{(1)},0}(u^{(1)})=(a_\infty^{(1)})^{-\frac{Q-2}{2}}\,\mathcal{E}_{1,0}\big((a_\infty^{(1)})^{\frac{Q-2}{4}}u^{(1)}\big)$, where $(a_\infty^{(1)})^{\frac{Q-2}{4}}u^{(1)}$ is a nontrivial weak solution of $\mathcal{L}_{\mathbb{H}^n}w=|w|^{2^\star-2}w$. Hence, by the Folland--Stein inequality (see Remark~\ref{rem_why_not_fractional}), we get
\[
    \mathcal{E}_{a_\infty^{(1)},0}(u^{(1)})
    \geq (a_\infty^{(1)})^{-\frac{Q-2}{2}}\cdot\frac{1}{Q}\mathcal{S}_{Q}^{Q/2}>0.
\]
Hence, the energy drops by a fixed positive amount:
\begin{equation*}
    \mathcal{E}_{a,0}(v_k) 
    \leq \mathcal{E}_{a,0}(\sigma_k) 
    - (a_\infty^{(1)})^{-\frac{Q-2}{2}}\cdot\frac{1}{Q}\mathcal{S}_{Q}^{Q/2}
    + o(1).
\end{equation*}

\smallskip

\noindent\textbf{$(PS)$ property for $\{v_k\}$.}
Fix $\varphi\in C_c^\infty(\mathbb{H}^n)$ and recall the rescaled test function $\varphi_k$ from \eqref{eq:rescaled_test},
which satisfies $\|\varphi_k\|_{S^1}=\|\varphi\|_{S^1}$. By Lemma~\ref{pp-lemma-3} and 
\eqref{claim_log+||},
we have $\varphi_k\rightharpoonup 0$ weakly in $S^1(\mathbb{H}^n)$, and in 
particular $\varphi_k\to 0$ a.e.\ up to a subsequence. Using 
\eqref{eq:vk_rescaled} and by change of variables, we can write
\begin{align}
     {}_{(S^1)'}\left\langle\mathcal{E}_{a,0}'(v_k),\varphi\right\rangle_{S^1}
    &= \langle v_k,\varphi\rangle_{S^1}
    - \int_{\mathbb{H}^n}a(\xi)|v_k|^{2^{\star}-2}v_k\,\varphi\,d\xi 
    \nonumber \\
    &= \langle\hat{v}_k,\varphi_k\rangle_{S^1}
    - \int_{\mathbb{H}^n}
    a_k(\eta)|\hat{v}_k|^{2^{\star}-2}\hat{v}_k\,\varphi_k\,d\eta 
    \nonumber \\
    &= \langle\tilde{\sigma}_k-u^{(1)},\varphi_k\rangle_{S^1}
    - \int_{\mathbb{H}^n}
    a_k|\tilde{\sigma}_k-u^{(1)}|^{2^{\star}-2}
    (\tilde{\sigma}_k-u^{(1)})\,\varphi_k\,d\eta.
    \label{eq:dE_vk}
\end{align}
Next, we claim
\begin{align}\label{eq:claim2}
  \begin{aligned}
       &\int_{\mathbb{H}^n}a_k|\tilde{\sigma}_k-u^{(1)}|^{2^{\star}-2}
    (\tilde{\sigma}_k-u^{(1)})\,\varphi_k\,d\eta\\
      &\hspace{1cm}  = \int_{\mathbb{H}^n}a_k|\tilde{\sigma}_k|^{2^{\star}-2}\tilde{\sigma}_k\,
    \varphi_k\,d\eta
    - \int_{\mathbb{H}^n}a_k|u^{(1)}|^{2^{\star}-2}u^{(1)}\,\varphi_k\,d\eta
    + o(1).
  \end{aligned}
\end{align}
Indeed, by the pointwise inequality: for $p>1$ and $s,t\in\mathbb{R}$,
\begin{equation*}
    \bigl||s+t|^{p-1}(s+t)-|s|^{p-1}s-|t|^{p-1}t\bigr|
    \leq C\bigl(|s|^{p-1}|t|+|s||t|^{p-1}\bigr),
\end{equation*}
and applying $p=2^{\star}-1$, one gets
\begin{multline*}
    \left|a_k|\tilde{\sigma}_k|^{2^{\star}-2}\tilde{\sigma}_k\,
    \varphi_k-a_k|\tilde{\sigma}_k-u^{(1)}|^{2^{\star}-2}
    (\tilde{\sigma}_k-u^{(1)})\,\varphi_k-a_k|u^{(1)}|^{2^{\star}-2}u^{(1)}\,\varphi_k\right|\\ \leq C\|a\|_{L^\infty}
    \Bigl(|\tilde{\sigma}_k-u^{(1)}|^{2^{\star}-2}|u^{(1)}|
    +|\tilde{\sigma}_k-u^{(1)}||u^{(1)}|^{2^{\star}-2}\Bigr)|\varphi_k|.
\end{multline*}
Now, we will be through if we establish
\[
\text{(i)}\int_{\hn}|\tilde{\sigma}_k-u^{(1)}|^{2^{\star}-2}|u^{(1)}|\, |\varphi_k| \,d\eta= o(1),
    \qquad
    \text{(ii)}\ \int_{\mathbb{H}^n}
    |\tilde{\sigma}_k-u^{(1)}||u^{(1)}|^{2^{\star}-2}\,|\varphi_k|\,d\eta=o(1).
\]
To see $(i)$,
\begin{align*}
     \int_{\hn}|\tilde{\sigma}_k-u^{(1)}|^{2^{\star}-2}|u^{(1)}||\varphi_k| \,d\eta= \int_{B(e,R)}|\tilde{\sigma}_k-u^{(1)}|^{2^{\star}-2}|u^{(1)}||\varphi_k| \,d\eta+\int_{\hn\setminus B(e,R)}|\tilde{\sigma}_k-u^{(1)}|^{2^{\star}-2}|u^{(1)}||\varphi_k| \,d\eta. 
\end{align*} 
Let $\varepsilon>0$ be given. Here in the above $R\equiv R(\varepsilon)>0$ is chosen so that $\int_{\hn\setminus B(e,R)}|u^{(1)}|^{2^{\star}}\,d\eta<\varepsilon.$ Now using the above and H\"older's inequality along-with Folland-Stein-Sobolev inequality
\begin{multline*}
     \int_{\hn\setminus B(e,R)}|\tilde{\sigma}_k-u^{(1)}|^{2^{\star}-2}|u^{(1)}||\varphi_k| \,d\eta \\ \leq  \left(\int_{\hn}|\tilde{\sigma}_k-u^{(1)}|^{2^{\star}}d\eta\right)^{\frac{2^{\star}-2}{2^{\star}}}\left(\int_{\hn\setminus B(e,R)}|u^{(1)}|^{2^{\star}}\,d\eta\right)^{\frac{1}{2^{\star}}}\|\varphi\|_{L^{2^{\star}}} =o(1).
\end{multline*}
Now using the fact that $\varphi_k\to 0$ a.e. on $\hn$ and Vitali's convergence theorem, we get
\begin{align*}
    \int_{B(e,R)}|\tilde{\sigma}_k-u^{(1)}|^{2^{\star}-2}|u^{(1)}||\varphi_k| \,d\eta = o(1).
\end{align*}
Similarly, $(ii)$ follows. This establishes \eqref{eq:claim2}.

Since $\|a_k\|_{L^\infty}=\|a\|_{L^\infty}$, 
$\|\varphi_k\|_{L^{2^{\star}}}=\|\varphi\|_{L^{2^{\star}}}$, and $\varphi_k\to 0$ a.e.,  H\"{o}lder's inequality and the dominated convergence theorem give us
\[
    \lim_{k\to\infty}
    \int_{\mathbb{H}^n}a_k|u^{(1)}|^{2^{\star}-2}u^{(1)}\,\varphi_k\,d\eta = 0.
\]
Applying the change of variables again, we go back to the original coordinates
\[
    \int_{\mathbb{H}^n}a_k|\tilde{\sigma}_k|^{2^{\star}-2}\tilde{\sigma}_k\,
    \varphi_k\,d\eta
    = \int_{\mathbb{H}^n}a(\xi)|\sigma_k|^{2^{\star}-2}\sigma_k\,\varphi\,d\xi.
\]
Substituting the above expression into \eqref{eq:claim2} and then into \eqref{eq:dE_vk}, and 
using \eqref{eq:inner_prod_identity} in the form 
$\langle\tilde{\sigma}_k,\varphi_k\rangle_{S^1}
=\langle\sigma_k,\varphi\rangle_{S^1}$ together with 
$\langle u^{(1)},\varphi_k\rangle_{S^1}\to 0$ (since $\varphi_k\rightharpoonup 0$ 
and $u^{(1)}\in S^1(\mathbb{H}^n)$), we obtain
\[
     {}_{(S^1)'}\left\langle\mathcal{E}_{a,0}'(v_k),\varphi\right\rangle_{S^1}
    = \langle\sigma_k,\varphi\rangle_{S^1}
    - \int_{\mathbb{H}^n}a(\xi)|\sigma_k|^{2^{\star}-2}\sigma_k\,\varphi\,d\xi
    + o(1)
    = {}_{(S^1)'}\left\langle\mathcal{E}_{a,0}'(\sigma_k),\varphi\right\rangle_{S^1} + o(1)
    = o(1),
\]
where the last equality uses $\mathcal{E}_{a,0}'(\sigma_k)\to 0$ from 
Step~III. By density this extends to all $\varphi\in S^1(\mathbb{H}^n)$, 
giving $\mathcal{E}_{a,0}'(v_k)\to 0$ in $(S^1(\mathbb{H}^n))'$.

\smallskip

\noindent\textbf{Conclusion.}
Combined with \eqref{eq:energy_vk}, $\{v_k\}$ is a $(PS)$ sequence for 
$\mathcal{E}_{a,0}$ with $v_k\rightharpoonup 0$, at the energy level
\[
    c - \mathcal{E}_{a,f}(u^{(0)}) 
    - \mathcal{E}_{a_\infty^{(1)},0}(u^{(1)}),
\]
which is strictly lower than the level $c-\mathcal{E}_{a,f}(u^{(0)})$ of 
$\{\sigma_k\}$ by the fixed positive amount $\mathcal{E}_{a_\infty^{(1)},0}(u^{(1)})$.
Since $\{v_k\}$ is again a $(PS)$ sequence for $\mathcal{E}_{a,0}$ with 
$v_k\rightharpoonup 0$, it has precisely the same structure as $\{\sigma_k\}$ 
and Steps~IV--VI may be applied iteratively. At the $j$-th iteration, a new parameter $\lambda_k^{(j)}$ satisfying 
$|\log\lambda_k^{(j)}|+\|\xi_k^{(j)}\|\to\infty$ and a new concentration 
point $\xi_k^{(j)}\in\mathbb{H}^n$ are extracted, a new nontrivial weak 
solution $u^{(j)}$ to \eqref{eq:limiting_eq_j} is identified, and the bubble 
$D_{\xi_k^{(j)},\lambda_k^{(j)}}u^{(j)}$ is subtracted. The energy level drops at each stage by at least
\begin{align}\label{eq_energy_drop}
  (a_\infty^{(j)})^{-\frac{Q-2}{2}}\cdot\frac{1}{Q}\mathcal{S}_{Q}^{Q/2}>0,
\end{align}
with equality when $u^{(j)}\geq 0$ (see Remark~\ref{rem_why_not_fractional}).
Since $\sup_k\|\sigma_k\|_{S^1(\mathbb{H}^n)}\leq C<\infty$, the energy 
levels are uniformly bounded below, so the iteration terminates after at most 
$m<\infty$ steps, at which point the remaining $(PS)$ sequence converges 
strongly to zero in $S^1(\mathbb{H}^n)$. The decomposition \eqref{eq_seq_decomp} then reads
\[
    u_k - \left(u^{(0)} + \sum_{j=1}^m 
    D_{\xi_k^{(j)},\lambda_k^{(j)}}u^{(j)}\right) \to 0
    \quad\text{strongly in }S^1(\mathbb{H}^n),
\]
establishing item~\eqref{it_profile_PS_thm} of Theorem~\ref{thm_PS_decomp}. The energy decoupling 
\eqref{eq:energy_decomp} follows by summing \eqref{eq:energy_vk} over all 
$m$ iterations and adding $\mathcal{E}_{a,f}(u^{(0)})$:
\[
    \mathcal{E}_{a,f}(u_k)
    = \mathcal{E}_{a,f}(u^{(0)})
    + \sum_{j=1}^m \mathcal{E}_{a_\infty^{(j)},0}(u^{(j)}) + o(1),
\]
establishing item~\eqref{item_en_dec_PS_thm} of Theorem~\ref{thm_PS_decomp}.

Finally, the asymptotic orthogonality \eqref{eq:orth} is a direct 
consequence of Lemma~\ref{pp-lemma-3}. Indeed, at each iteration the 
rescaled sequence $\hat{v}_k^{(j-1)}$ (obtained by applying 
$D_{\xi_k^{(j)},\lambda_k^{(j)}}$ to the remainder) converges weakly to 
the nontrivial profile $u^{(j)}$, while the remainder itself satisfies 
$v_k^{(j-1)}\rightharpoonup 0$. This forces
\[
    \langle D_{\xi_k^{(i)},\lambda_k^{(i)}}u,\,
    D_{\xi_k^{(j)},\lambda_k^{(j)}}v\rangle_{S^1}\to 0
    \quad\text{for all }u,v\in S^1(\mathbb{H}^n),\quad i\neq j,
\]
which by Lemma~\ref{pp-lemma-3} is equivalent to \eqref{eq:orth}, 
establishing item~\eqref{ite_aymp_ortho} of Theorem~\ref{thm_PS_decomp} and completing 
the proof.
\end{proof} 
\begin{remark}\label{rem_why_not_fractional}
We would like to point out that the lower bound \eqref{eq_energy_drop} does not use the classification of extremals. Indeed, let $w\not\equiv 0$ be a solution of $\mathcal{L}_{\mathbb{H}^n}w=a_\infty|w|^{2^\star-2}w$, and set $\tilde{w}:=a_\infty^{\frac{Q-2}{4}}w$. Then $\|\tilde{w}\|_{S^1}^2=\|\tilde{w}\|_{L^{2^\star}}^{2^\star}\leq \big(\mathcal{S}_Q^{-1}\|\tilde{w}\|_{S^1}^2\big)^{2^\star/2}$, whence $\|\tilde{w}\|_{S^1}^2\geq \mathcal{S}_Q^{Q/2}$, and consequently $\mathcal{E}_{1,0}(\tilde{w})=\frac{1}{Q}\|\tilde{w}\|_{S^1}^2\geq \frac{1}{Q}\mathcal{S}_Q^{Q/2}$. Moreover, if $w$ changes sign, then testing the equation with $\tilde{w}_\pm$, we get the same bound for both $\tilde{w}_+$ and $\tilde{w}_-$; therefore, a sign-changing profile carries energy at least $\frac{2}{Q}\mathcal{S}_Q^{Q/2}$. The Heisenberg structure is only used for the case of equality. More precisely, if $u^{(j)}\geq 0$, then $u^{(j)}>0$ by Proposition~\ref{thm_weak_max_prin}, and hence it is a Jerison--Lee bubble by Theorem~\ref{thm_JL_bubble}, which in turn implies that the energy drop is exactly the explicit number in \eqref{eq_energy_drop}. The same argument gives the lower bound on any Carnot group, with $\mathcal{S}_Q$ replaced by the best Folland--Stein constant of the group; this constant is attained \cite{Vas06}, but it is not explicitly known, and therefore the minimal drop is not explicit there. Finally, for the CR-covariant fractional powers of $\mathcal{L}_{\mathbb{H}^n}$, the extremals are known \cite{FL12}, but one would need a fractional analogue of Theorem~\ref{thm_hardy_inter}.
\end{remark}

We now show that the energy levels in Theorem~\ref{thm_PS_decomp} are optimal, that is, compactness indeed fails at each of them. This is the content of the following theorem, which also gives the second part of Theorem~\ref{thm_exact_levels}.

\begin{theorem}\label{thm_sharp_levels}
Let $u^{(0)}$ be a weak solution of \eqref{eqn_pert_Hn} (without the positivity condition), let $m\geq 1$, and let $\alpha_1,\ldots,\alpha_m\in a(\mathbb{H}^n)\cup\{1\}$. Then there exists a $(PS)_c$ sequence $\{u_k\}$ for $\mathcal{E}_{a,f}$ at the level
\[
    c=\mathcal{E}_{a,f}(u^{(0)})+\frac{1}{Q}\,\mathcal{S}_Q^{Q/2}\sum_{j=1}^m \alpha_j^{-\frac{Q-2}{2}},
\]
such that $u_k\rightharpoonup u^{(0)}$ weakly in $S^1(\mathbb{H}^n)$, but no subsequence of $\{u_k\}$ converges strongly. Moreover, if $u^{(0)}\geq 0$, then the sequence can be chosen to be nonnegative.
\end{theorem}
\begin{proof}
For simplicity, we write $p:=2^\star-1$ and $D_\theta:=D_{\xi,\lambda}$ for $\theta=(\xi,\lambda)$. The idea is to add to $u^{(0)}$ suitable bubbles which do not interact with each other. More precisely, for each $j$, we choose $\theta_k^{(j)}=(\zeta_j,k^{-j})$ if $\alpha_j=a(\zeta_j)$ for some $\zeta_j\in\mathbb{H}^n$, and $\theta_k^{(j)}=(e,k^{j})$ otherwise, in which case $\alpha_j=1$. Let us define
\[
    w_k^{(j)}:=\alpha_j^{-\frac{Q-2}{4}}D_{\theta_k^{(j)}}U,\qquad u_k:=u^{(0)}+\sum_{j=1}^m w_k^{(j)}.
\]
Then $\mathcal{L}_{\mathbb{H}^n}w_k^{(j)}=\alpha_j\,(w_k^{(j)})^{p}$ and $\|w_k^{(j)}\|_{S^1}^2=\alpha_j^{-\frac{Q-2}{2}}\mathcal{S}_Q^{Q/2}$. Since the exponents $\pm j$ are pairwise distinct, any two of the sequences $(e,1),\theta_k^{(1)},\ldots,\theta_k^{(m)}$ satisfy condition (ii) of Lemma~\ref{pp-lemma-3}. Consequently, all the cross terms $\langle u^{(0)},w_k^{(j)}\rangle_{S^1}$ and $\langle w_k^{(i)},w_k^{(j)}\rangle_{S^1}$, $i\neq j$, tend to zero, and $w_k^{(j)}\rightharpoonup 0$, which in turn implies that $u_k\rightharpoonup u^{(0)}$ and ${}_{(S^1)'}\langle f,w_k^{(j)}\rangle_{S^1}\to 0$. We now complete the proof in four steps.

\emph{Step 1: interaction estimate.} Let $\alpha,\beta>0$ with $\alpha+\beta=2^\star$, let $\varphi,\psi\in L^{2^\star}(\mathbb{H}^n)$, and let $\{\theta_k\},\{\theta_k'\}$ satisfy condition (ii) of Lemma~\ref{pp-lemma-3}. We claim that $\int_{\mathbb{H}^n}|D_{\theta_k}\varphi|^{\alpha}|D_{\theta'_k}\psi|^{\beta}\,d\xi\to 0$. Since $\alpha+\beta=2^\star$, the integral does not change if we apply the same $D_\theta^{-1}$ to both factors, and hence, by the computation in the proof of Lemma~\ref{pp-lemma-3}, we may assume that $\theta_k=(e,1)$ and $\theta_k'=(\eta_k,r_k)$ with $|\log r_k|+\|\eta_k\|\to\infty$. Moreover, since $D_{\theta}$ is an isometry of $L^{2^\star}(\mathbb{H}^n)$, by H\"older's inequality it is enough to prove the claim for $\varphi,\psi\in C_c(\mathbb{H}^n)$, say supported in $B(e,R)$. Passing to a subsequence, we have three possibilities. If $r_k\to 0$, then $D_{\theta'_k}\psi$ is supported in $B(\eta_k,r_kR)$, and therefore the integral is at most $\|\varphi\|_{L^\infty}^{\alpha}|B(\eta_k,r_kR)|^{1-\beta/2^\star}\|\psi\|_{L^{2^\star}}^{\beta}\to 0$. If $r_k\to\infty$, then $\|D_{\theta'_k}\psi\|_{L^\infty}=r_k^{-\frac{Q-2}{2}}\|\psi\|_{L^\infty}\to 0$. Finally, if $r_k$ stays bounded away from $0$ and $\infty$ and $\|\eta_k\|\to\infty$, then the supports of $\varphi$ and $D_{\theta'_k}\psi$ are eventually disjoint. This proves our claim.

\emph{Step 2: behaviour of the coefficient.} Writing $\theta_k^{(j)}=(\xi_k^{(j)},\lambda_k^{(j)})$ and using the change of variables $\xi=\xi_k^{(j)}\circ\delta_{\lambda_k^{(j)}}(\eta)$, we get, for every $q\geq 1$,
\[
    \int_{\mathbb{H}^n}|a(\xi)-\alpha_j|^{q}\,|w_k^{(j)}|^{2^\star}\,d\xi
    =\alpha_j^{-\frac{Q}{2}}\int_{\mathbb{H}^n}\big|a\big(\xi_k^{(j)}\circ\delta_{\lambda_k^{(j)}}(\eta)\big)-\alpha_j\big|^{q}\,U(\eta)^{2^\star}\,d\eta\longrightarrow 0
\]
by the dominated convergence theorem. Indeed, if $\theta_k^{(j)}=(\zeta_j,k^{-j})$, then $\zeta_j\circ\delta_{k^{-j}}(\eta)\to\zeta_j$ and $a$ is continuous, while if $\theta_k^{(j)}=(e,k^{j})$, then $\|\delta_{k^{j}}(\eta)\|=k^j\|\eta\|\to\infty$ for $\eta\neq e$, and $\alpha_j=1$.

\emph{Step 3: energy and derivative.} We recall the following elementary inequality: for all real numbers $b_0,\ldots,b_m$, there exists $C=C(m,Q)>0$ such that
\begin{equation}\label{eq_elem_ineq}
    \Big|\big|\textstyle\sum_{i} b_i\big|^{2^\star}-\sum_{i}|b_i|^{2^\star}\Big|
    +\Big|\big|\textstyle\sum_{i} b_i\big|^{p-1}\sum_{i} b_i-\sum_{i}|b_i|^{p-1}b_i\Big|
    \leq C\sum_{i\neq l}\big(|b_i|^{p}|b_l|+|b_i|^{p-1}|b_l|\big)
\end{equation}
 Throughout this step, we apply \eqref{eq_elem_ineq} with $b_0=u^{(0)}$ and $b_j=w_k^{(j)}$, $j=1,\ldots,m$.

We first show that $\mathcal{E}_{a,f}(u_k)\to c$. Since $a\in L^\infty(\mathbb{H}^n)$, the inequality \eqref{eq_elem_ineq} together with Step~1 for $(\alpha,\beta)=(p,1)$ shows that the cross terms vanish in the limit. Using then Step~2 with $q=1$, and recalling that $\int_{\mathbb{H}^n}|w_k^{(j)}|^{2^\star}d\xi=\alpha_j^{-\frac{Q}{2}}\mathcal{S}_Q^{Q/2}$, we obtain
\begin{align*}
    \int_{\mathbb{H}^n}a|u_k|^{2^\star}d\xi
    =\int_{\mathbb{H}^n}a|u^{(0)}|^{2^\star}d\xi+\sum_{j=1}^m\int_{\mathbb{H}^n}a|w_k^{(j)}|^{2^\star}d\xi+o(1) &=\int_{\mathbb{H}^n}a|u^{(0)}|^{2^\star}d\xi+\sum_{j=1}^m\alpha_j\int_{\mathbb{H}^n}|w_k^{(j)}|^{2^\star}d\xi+o(1)\\
    &=\int_{\mathbb{H}^n}a|u^{(0)}|^{2^\star}d\xi+\mathcal{S}_Q^{Q/2}\sum_{j=1}^m\alpha_j^{-\frac{Q-2}{2}}+o(1).
\end{align*}
On the other hand, by the orthogonality of the cross terms, $\|u_k\|_{S^1}^2=\|u^{(0)}\|_{S^1}^2+\mathcal{S}_Q^{Q/2}\sum_{j}\alpha_j^{-\frac{Q-2}{2}}+o(1)$ and ${}_{(S^1)'}\langle f,u_k\rangle_{S^1}\to{}_{(S^1)'}\langle f,u^{(0)}\rangle_{S^1}$. Therefore, since $\frac12-\frac1{2^\star}=\frac1Q$, we get $\mathcal{E}_{a,f}(u_k)=\mathcal{E}_{a,f}(u^{(0)})+\frac{1}{Q}\mathcal{S}_Q^{Q/2}\sum_{j}\alpha_j^{-\frac{Q-2}{2}}+o(1)\to c$. Next, we show that $\mathcal{E}_{a,f}'(u_k)\to0$ in $(S^1(\mathbb{H}^n))'$. Using the equations satisfied by $u^{(0)}$ and $w_k^{(j)}$, we can write ${}_{(S^1)'}\langle\mathcal{E}_{a,f}'(u_k),\phi\rangle_{S^1}=-\int_{\mathbb{H}^n}g_k\,\phi\,d\xi$ for all $\phi\in S^1(\mathbb{H}^n)$, where
\[
    g_k:=a\Big(|u_k|^{p-1}u_k-|u^{(0)}|^{p-1}u^{(0)}-\sum_{j=1}^m(w_k^{(j)})^{p}\Big)+\sum_{j=1}^m\big(a-\alpha_j\big)(w_k^{(j)})^{p}.
\]
By H\"older's inequality and the Folland--Stein inequality, $\|\mathcal{E}_{a,f}'(u_k)\|_{(S^1)'}\leq \mathcal{S}_Q^{-1/2}\|g_k\|_{L^{q'}}$ with $q'=\frac{2Q}{Q+2}$, and therefore it is enough to show that both parts of $g_k$ tend to zero in $L^{q'}(\mathbb{H}^n)$. By \eqref{eq_elem_ineq}, the first part of $g_k$ is bounded pointwise by $C\|a\|_{L^\infty}\sum_{i\neq l}|b_i|^{p-1}|b_l|$, and $\||b_i|^{p-1}b_l\|_{L^{q'}}^{q'}=\int_{\mathbb{H}^n}|b_i|^{(p-1)q'}|b_l|^{q'}\,d\xi\to0$ by Step~1 with $(\alpha,\beta)=((p-1)q',q')$, which is admissible since $(p-1)q'+q'=pq'=2^\star$. For the second part, since $pq'=2^\star$, we have $\|(a-\alpha_j)(w_k^{(j)})^{p}\|_{L^{q'}}^{q'}=\int_{\mathbb{H}^n}|a-\alpha_j|^{q'}|w_k^{(j)}|^{2^\star}d\xi$, which tends to zero by Step~2 with $q=q'$. Hence $\{u_k\}$ is a $(PS)_c$ sequence.

\emph{Step 4: lack of compactness.} Suppose that a subsequence of $\{u_k\}$ converges strongly. Then its limit must be the weak limit $u^{(0)}$. However, $\|u_k-u^{(0)}\|_{S^1}^2=\|\sum_jw_k^{(j)}\|_{S^1}^2\to\mathcal{S}_Q^{Q/2}\sum_j\alpha_j^{-\frac{Q-2}{2}}>0$, which is a contradiction. Finally, since $w_k^{(j)}>0$, we have $u_k\geq0$ whenever $u^{(0)}\geq0$. This completes the proof.
\end{proof}

We are now in a position to prove Theorem~\ref{thm_exact_levels}.

\begin{proof}[\textbf{Proof of Theorem~\ref{thm_exact_levels}}]
Let $\{u_k\}$ be a nonnegative $(PS)_c$ sequence which is not relatively compact. Then, by Theorem~\ref{thm_PS_decomp}, up to a subsequence, $\{u_k\}$ decomposes with $m\geq1$, $u^{(0)}\geq0$ and nonnegative profiles $u^{(j)}$, and moreover $a_\infty^{(j)}\in a(\mathbb{H}^n)\cup\{1\}$ by \eqref{eq:a_infty}. Using Theorem~\ref{thm_JL_bubble}, we get $u^{(j)}=(a_\infty^{(j)})^{-\frac{Q-2}{4}}D_{\eta^{(j)},\lambda^{(j)}}U$, and hence $\mathcal{E}_{a_\infty^{(j)},0}(u^{(j)})=(a_\infty^{(j)})^{-\frac{Q-2}{2}}\frac{1}{Q}\mathcal{S}_Q^{Q/2}$. Therefore, \eqref{eq:energy_decomp} implies that $c\in\Lambda_{a,f}$. Conversely, if $c\in\Lambda_{a,f}$, then Theorem~\ref{thm_sharp_levels} provides a nonnegative $(PS)_c$ sequence which converges weakly, but not strongly. This completes the proof.
\end{proof}

\section{Existence of multiple positive solutions}\label{mul sol cry}
In this section, we prove Theorem~\ref{thm_multi_pos_sol}. We begin by recalling the strong maximum principle on $\mathbb{H}^n$. For classical solutions, a proof can be found in \cite[Corollary~3.1]{Bo69}, while for weak solutions it follows from the weak Harnack inequality \cite[Proposition~2.2]{BGV26} together with the fact that any two points of $\mathbb{H}^n$ can be joined by a horizontal path. For the sake of brevity, we omit the details.

\begin{proposition}\label{thm_weak_max_prin}
 Let $u\in S^1(\hn)$ be a nonnegative function satisfies
    \[
     \mathcal{L}_{\hn}u (z,t)\geq 0 \quad \text{in } \quad \mathbb{H}^n.
    \]
     Then either $u>0$ a.e. in $\mathbb{H}^n$ or $u= 0$ a.e. in $\mathbb{H}^n$.   
\end{proposition}

The first solution is a local minimizer with negative energy. The second is obtained by the mountain-pass theorem: the smallness of $f$ keeps the mountain-pass level below the first bubble level, and Theorem~\ref{thm_PS_decomp} then excludes bubbling. We first consider the auxiliary functional
\begin{equation}\label{eqref_def_Jf}
    J_{f}(v)\coloneqq \frac{1}{2}\|\nabla_{\mathbb{H}^n}v\|_{L^2(\mathbb{H}^n)}^2 - \frac{1}{2^{\star}}\int_{\mathbb{H}^n}a(\xi)v_{+}^{2^{\star}}\,d\xi - {}_{(S^1)'}\langle f,v\rangle_{S^1},
\end{equation}
where $v_{+}\coloneqq \max \{v,0\}$ and $v_{-}\coloneqq \max\{-v,0\}$. If $u\in S^1(\mathbb{H}^n)$ is a critical point of $J_{f}$, then $u$ weakly solves the positive problem
\begin{equation}\label{pos_eqn_pert_Hn}\tag{$\wp_{+}$}
\begin{cases}
   \mathcal{L}_{\mathbb{H}^n} u =  a(\xi)u_{+}^{2^{\star}-1} + f(\xi) \quad &\text{in } \mathbb{H}^n, \\
    u \in S^1(\mathbb{H}^n).
\end{cases}
\end{equation}
Equivalently, for all test functions $\phi\in S^1(\mathbb{H}^n)$, we have
\begin{align}\label{aux_var_id}
    \int_{\mathbb{H}^n}\nabla_{\mathbb{H}^n}u\cdot\nabla_{\mathbb{H}^n}\phi \,d\xi = \int_{\mathbb{H}^n}a(\xi)u_{+}^{2^{\star}-1}\phi\,d\xi + {}_{(S^1)'}\langle f,\phi \rangle_{S^1}. 
\end{align}
Taking $\phi = u_{-}\in S^1(\mathbb{H}^n)$ as the test function in \eqref{aux_var_id} and recalling that $f$ is a nonnegative functional, we obtain
\begin{equation*}
    0 \geq -\|u_{-}\|_{S^1(\mathbb{H}^n)}^2 = {}_{(S^1)'}\langle f,u_{-}\rangle_{S^1} \geq 0.
\end{equation*}
This forces $u_{-} \equiv 0$, which implies $u\geq 0$ almost everywhere in $\mathbb{H}^n$. Applying the strong maximum principle for subelliptic superharmonic functions (see Proposition~\ref{thm_weak_max_prin}), we then deduce that $u>0$ strictly almost everywhere.

\begin{remark}\label{rmk_equi_cric_pt}
    From the above discussion, it is evident that any critical point $u\in S^1(\mathbb{H}^n)$ of the auxiliary functional $J_{f}$ is strictly positive, and therefore constitutes a classical weak solution to the original problem \eqref{eqn_pert_Hn}.
\end{remark}

To rigorously establish the existence of two distinct critical points for $J_{f}$, we employ a variational partitioning strategy inspired by the Euclidean frameworks of \cite{BP20, BCG23} (see also \cite{Loi20} for bounded domains), which we adapt here to the whole Heisenberg group. We define the fiber map $\Psi : S^1(\mathbb{H}^n)\to\mathbb{R}$ by
\begin{equation*}
    \Psi(u)\coloneqq \|u\|_{S^1(\mathbb{H}^n)}^2 - p \int_{\mathbb{H}^n}a(\xi)|u|^{2^{\star}}\,d\xi, \quad \text{for all } u\in S^1(\mathbb{H}^n),
\end{equation*}
where $p\coloneqq 2^{\star}-1$. Using this $C^2$ functional, we partition the space $S^1(\mathbb{H}^n)$ into three mutually disjoint subsets:
\begin{align*}
    \Upsilon_1 &\coloneqq \left\{u\in S^1(\mathbb{H}^n)\,:\,u=0 \text{ or }\Psi (u)>0\right\},\\
    \Upsilon   &\coloneqq \left\{u\in S^1(\mathbb{H}^n)\setminus \{0\}\,:\,\Psi (u)=0\right\}, \\
    \Upsilon_2 &\coloneqq \left\{u\in S^1(\mathbb{H}^n)\,:\,\Psi (u)<0\right\}.
\end{align*}
These sets are homeomorphic to the open unit ball, the unit sphere, and the exterior of the unit ball in $S^1(\mathbb{H}^n)$, respectively. We then define the corresponding critical energy levels:
\begin{equation}\label{cric_lev}
    c_0\coloneqq \inf_{u \in \Upsilon_1} J_{f}(u),\qquad c_1\coloneqq \inf_{u \in \Upsilon} J_{f}(u).
\end{equation}

\begin{remark}\label{bdd_away}
    For any $u\in \Upsilon$, the Folland--Stein--Sobolev inequality yields
    \begin{equation*}
        \|u\|_{S^1(\mathbb{H}^n)}^2 = p\int_{\mathbb{H}^n}a|u|^{2^{\star}}d\xi \leq p\|a\|_{L^\infty}\mathcal{S}_Q^{-\frac{2^{\star}}{2}} \|u\|_{S^1(\mathbb{H}^n)}^{2^{\star}}.
    \end{equation*}
    Since $2^\star>2$, both $\|u\|_{S^1(\mathbb{H}^n)}$ and $\int_{\mathbb{H}^n}a|u|^{2^{\star}}d\xi$ are bounded away from $0$ uniformly in $u\in\Upsilon$. 
\end{remark}

\begin{remark}\label{rmk_psi}
    For any $u \in S^1(\mathbb{H}^n)$ and $t>0$, we have
    \begin{equation*}
        \Psi (tu) = t^2\|u\|_{S^1(\mathbb{H}^n)}^2 - t^{2^{\star}} p\int_{\mathbb{H}^n}a|u|^{2^{\star}}d\xi.
    \end{equation*}
    Hence, for every $u\neq0$, there exists a unique $t_u>0$ such that $t_u u\in\Upsilon$. Moreover, if $u\in\Upsilon$, then $\Psi(tu)=(t^2-t^{2^{\star}})\|u\|_{S^1(\mathbb{H}^n)}^2$, and therefore $tu\in\Upsilon_1$ for $t\in(0,1)$ and $tu\in\Upsilon_2$ for $t>1$.
\end{remark}

Throughout this section, we denote by $C_0=\frac{4}{Q+2}\,p^{-\frac{Q-2}{4}}$ the constant appearing in \eqref{eq:f_small}.

\begin{lemma}\label{lem_strict_less}
Let $C_0$ be as in \eqref{eq:f_small}, and let $c_0$ and $c_1$ be defined by \eqref{cric_lev}. If
\begin{equation}\label{J1:3}
\inf_{\substack{u\in S^1(\hn)\\ \int_{\hn}a|u|^{2^{\star}}d\xi=1}}
\left\{
C_0\|u\|_{S^1}^{\frac{Q+2}{2}}
-{}_{(S^1)'}\langle f,u\rangle_{S^1}
\right\}>0,
\end{equation}
then $c_0<c_1$.
\end{lemma}

\begin{proof}
Define a modified functional $\tilde{I}_{f}:S^1(\hn)\to \R$ by,
\begin{equation}\label{def_I_tild_f}
    \tilde{I}_{f}(u)\coloneqq \frac{1}{2}\|u\|_{S^1}^2 -\frac{1}{2^{\star}}\int_{\hn}a|u|^{2^{\star}}d\xi-{}_{(S^1)'}{\langle}f,u\rangle_{S^1}.
\end{equation}
{\bf Step 1}: First, we assert that there exists $\delta>0$ {(independent of $u$)} such that
		$$\frac{{\rm d}}{{\rm d}t}\bigg{|}_{t=1}\tilde{I}_{f}(tu)\geq \delta \quad\forall\,\,  u\in \Upsilon.$$ 
Doing a straightforward computation, it is easy to see that for any $u\in \Upsilon$,
\begin{align}\label{15-4-4}
\frac{{\rm d}}{{\rm d}t}\bigg{|}_{t=1} \tilde{I}_{f}(tu)=\frac{4}{Q+2}\|u\|_{S^1}^2-{}_{(S^1)'}{\langle}f,u{\rangle}_{S^1}
=C_0\frac{\|u\|^\frac{(Q+2)}{2}_{S^1}}{\left(\int_{\hn}a|u|^{2^{\star}}d\xi\right)^\frac{Q-2}{4}}- {}_{(S^1)'}{\langle}f,u{\rangle}_{S^1}.
\end{align}
Further, $\eqref{J1:3}$ implies there  exists $C>0$ such that
\begin{equation}\label{15-4-3}
	\inf_{\substack{u\in S^1(\hn),\\ \int_{\hn}a|u|^{2^{\star}} d\xi=1}}\left\{C_0\|u\|_{S^1}^\frac{(Q+2)}{2}
	-{}_{(S^1)'}{\langle}f,u{\rangle}_{S^1}\right\}\geq C.
\end{equation}

Now, 
\begin{align}
\eqref{15-4-3}&\Longleftrightarrow C_0\frac{\|u\|^\frac{Q+2}{2}_{S^1}}{(\int_{\hn}a|u|^{2^{\star}}d\xi)^\frac{Q-2}{4}}-{}_{(S^1)'}{\langle}f,u{\rangle}_{S^1} \geq C, \text{ with } \int_{\hn}a|u|^{2^{\star}}d\xi=1.\nonumber\\
&\Longleftrightarrow C_0\frac{\|u\|^\frac{Q+2}{2}_{S^1}}{(\int_{\hn}a|u|^{2^{\star}}d\xi)^\frac{Q-2}{4}}-{}_{(S^1)'}{\langle}f,u{\rangle}_{S^1} \geq C\left(\int_{\hn}a|u|^{2^{\star}}d\xi\right)^{\frac{1}{2^{\star}}}, \; \forall\, u\in S^1(\hn)\setminus\{0\}.\nonumber
\end{align}
From Remark~\ref{bdd_away} we can see $\int_{\hn}a|u|^{2^{\star}}\,d\xi$ is bounded away from $0$ for all $u\in \Upsilon.$  Substituting the above estimate into \eqref{15-4-4} completes Step~1.

\noindent {\bf Step 2:}  Let $\{u_n\}\subset\Upsilon$ be a minimizing sequence for $J_f$, namely, $J_f(u_n)\to c_1$ and
$ 
\|u_n\|_{S^1}^2=p\int_{\hn}a|u_n|^{2^{\star}}\,d\xi
$ for every $n\in\mathbb{N}$. Hence, for large $n$,  
\begin{align*}c_1+o(1)\geq J_{f}(u_n)\geq \tilde{I}_{f}(u_n)&\geq \left(\frac{1}{2}-\frac{1}{p(p+1)}\right)\|u_n\|^2_{S^1}- \|f\|_{(S^1)'}\|u_n\|_{S^1}.
  \end{align*}
The first two inequalities imply that $\left\{\tilde{I}_{f}(u_n)\right\}$ is a bounded sequence. Moreover, using the last inequality we get $\{\|u_n\|_{S^1}\}$ is bounded and hence  $\left\{\int_{\hn}a|u_n|^{2^{\star}}d\xi\right\}$ is also bounded. In this step, for large $n$, we prove the existence of a unique $t_n\in (0,1)$ such that 
\begin{align}\label{t_nexistence}
    \frac{{\rm d}}{{\rm d}t}\bigg{|}_{t=t_n}\tilde{I}_{f}(t u_n)=0.
\end{align}
First, we claim that
\begin{align}\label{c_0 negative}
    c_0<0.
\end{align}
To prove \eqref{c_0 negative}, it is enough to find a $v\in \Upsilon_1$ with $J_{f}(v)<0$. In view of Remark \ref{rmk_psi}, we choose $u \in \Upsilon$ such that ${}_{(S^1)'}{\langle}f,u{\rangle}_{S^1}>0$. Further, 
$$J_{f}(\tilde{t}u)= {\tilde{t}}^2 \left(\frac{p}{2}-\frac{(\tilde{t})^{2^{\star}-2}}{2^{\star}}\right) \int_{\hn}a|u|^{2^{\star}}\,d\xi-\tilde{t}{}_{(S^1)'}{\langle}f,u{\rangle}_{S^1} <0,$$
for $\tilde{t} << 1$, and using Remark \ref{rmk_psi} we also have $\tilde{t} u\in \Upsilon_1$. Hence, \eqref{c_0 negative} holds.

By definition, we have $c_0\leq c_1$, and if $c_1\geq 0$, then due to \eqref{c_0 negative}, there is nothing to prove. So for large $n \in \mathbb{N}$, we have
		$$0>J_{f}(u_n)\geq \left(\frac{1}{2}-\frac{1}{p(p+1)}\right)\|u_n\|^2_{S^1}-{}_{(S^1)'}{\langle}f,u_n{\rangle}_{S^1},$$
which implies ${}_{(S^1)'}{\langle}f,u_n{\rangle}_{S^1}> 0$. Consequently, for large $n \in \mathbb{N}$, 
\begin{equation}\label{est_I_tild_f}
    \frac{{\rm d}}{{\rm d}t}\tilde{I}_{f}(tu_n) \leq p t \int_{\hn} a|u_n|^{2^{\star}}\,d\xi - {}_{(S^1)'}{\langle}f,u_n{\rangle}_{S^1} <0, \text{ for $t<<1$.}
\end{equation}

Therefore, in view of Step~1, there exists $t_n \in (0,1)$ such that \eqref{t_nexistence} holds. Since for all $u\in \Upsilon$, the function $\frac{{\rm d}}{{\rm d}t}\tilde{I}_{f}(tu)$ is strictly increasing for $t\in [0,1)$, we conclude that $t_n$ is unique.

\noindent {\bf Step 3:} In this step, we show that 	
\begin{equation}\label{J3}
\liminf_{n\rightarrow\infty} \left\{ \tilde{I}_{f}(u_n)-\tilde{I}_{f}(t_nu_n) \right\}>0,
\end{equation}
{where $\{u_n\}$ is the same minimizing sequence of $c_1$ as chosen in Step 2.} We denote $h_n(t):= \tilde{I}_{f}(tu_n)$. Then for each $n\in\mathbb{N}$, $h_n\in C^2(\R)$ and has the following properties (by Step~1 and Step~2):
\begin{align*}
&(a) \, h_n'(1)\geq \delta,\\
&(b) \, \text{there exists } {0<t_n^*<<1}, \text{ such that }h_n'(t_n^*)<0,\\ 
&(c) \, h_n' \text{ is increasing in }[0,1],\\
&(d) \, \text{for each }n\in\mathbb{N}, \text{ there exists a unique }t_n\in(0,1-2s_n) \text{ such that }h_n'(t_n)=0.   
\end{align*}
In view of Step 1 and Step 2, we choose $0<s_n<\frac{1-t_n}{2}$ such that $ h_n'(t) \geq \tfrac{\delta}{2}$ for $t\in[1-s_n,\;1]$. 
Therefore, using the fundamental theorem of calculus,
\begin{equation*}
\tilde{I}_{f}(u_n)-\tilde{I}_{f}(t_nu_n)=h_n(1)-h_n(t_n)=\displaystyle \int_{t_n}^{1}h_n'(s) ds\geq \int_{1-s_n}^{1} h_n'(s)\,ds\geq \frac{\delta s_n}{2}.
\end{equation*}
To establish \eqref{J3}, it is enough to show that $s_n$ can be chosen independently of $n\in\mathbb{N}$. This is possible as $\{u_n\}$ is bounded in $S^1(\hn),$ so that
\begin{align*}
 \left|h_n^{''}(t)\right| = \left|\|u_n\|_{S^1}^2 - pt^{p-1}\int_{\hn}a|u_n|^{2^{\star}}\,d\xi\right| =  \left| 1-t^{p-1} \right| \|u_n\|_{S^1}^2\leq C,
\end{align*}
for all $n \in \mathbb{N}$ and $t\in[0,1]$.
	
\noindent {\bf Step 4:} 
By the definitions of $J_f$ and $\tilde{I}_f$, $\frac{{\rm d}}{{\rm d}t}J_f(tu)\geq\frac{{\rm d}}{{\rm d}t}\tilde{I}_f(tu)$ for all $u\in S^1(\hn)$ and $t>0$. Hence,
\begin{align*}
J_{f}(u_n)-J_{f}(t_nu_n)= \int_{t_n}^{1}\frac{{\rm d}}{{\rm d}t}\left(J_{f}(tu_n)\right)\,dt \geq \int_{t_n}^{1}\frac{{\rm d}}{{\rm d}t}\left(\tilde{I}_{f}(tu_n)\right)\,dt 
= \tilde{I}_{f}(u_n)-\tilde{I}_{f}(t_nu_n).
\end{align*}
Since $\{u_n\}$ minimizes $J_f$ on $\Upsilon$, $t_nu_n\in\Upsilon_1$, and using \eqref{J3} , we conclude
$$c_0 = \inf_{u\in \Upsilon_1}J_{f}(u) \leq \liminf_{n \to \infty} J_{f}(t_n u_n) < \liminf_{n \to \infty} J_{f}(u_n) = \inf_{u\in \Upsilon}J_{f}(u)= c_1.$$
\end{proof}

\begin{proposition}\label{fir-cri-pt}
Assume that \eqref{J1:3} holds. Then $J_{f}$ has a critical point $u_0\in \Upsilon_1$ with $J_{f}(u_0)=c_0<0$. In particular, $u_0\in S^1(\hn)$ is a positive weak solution to \eqref{eqn_pert_Hn} whose energy is negative.
\end{proposition}

\begin{proof} We decompose the proof into a few steps.
		
\noindent {\bf Step~1:} 
In this step, we prove that $c_0>-\infty$. Since $J_f(u)\geq \tilde{I}_f(u)$,it suffices to show that $\tilde{I}_f$ is bounded from below. By the definition of $\Upsilon_1$, for every $u\in\Upsilon_1$ we have
\begin{equation}\label{16-4-1}
\tilde{I}_f(u)
\geq
\left(
\frac{1}{2}-\frac{1}{p(p+1)}
\right)
\|u\|_{S^1}^2
-
\|f\|_{(S^1)'}\|u\|_{S^1}.
\end{equation}
The right-hand side of \eqref{16-4-1} is a quadratic polynomial in $\|u\|_{S^1}$ with positive leading coefficient. Hence, it is bounded from below, and therefore so is $\tilde{I}_f$. Consequently, $c_0>-\infty$, which completes Step~1.

\noindent {\bf Step~2}: 
In this step, we establish the existence of a bounded nonnegative $(PS)$ sequence $\{u_n\}\subset\Upsilon_1$ for $J_f$ at level $c_0$. Let $\{u_n\}\subset\overline{\Upsilon}_1$ satisfy $J_f(u_n)\to c_0$. Because of assumption \eqref{J1:3}, Lemma~\ref{lem_strict_less} holds and it gives $c_0<c_1$, and thereby we may assume that $\{u_n\}\subset\Upsilon_1$. By Ekeland's variational principle, after passing to a subsequence, $\{u_n\}$ is a $(PS)_{c_0}$ sequence for $J_f$ in $\Upsilon_1$. Since $J_f(u)\geq\tilde{I}_f(u)$, estimate \eqref{16-4-1} shows that $\{u_n\}$ is bounded in $S^1(\hn)$. Thus, up to a subsequence, $u_n\rightharpoonup u_0$ in $S^1(\hn)$ and $u_n\to u_0$ a.e. in $\hn$. In particular, $(u_n)_\pm\to(u_0)_\pm$ a.e. in $\hn$. Since $f$ is a nonnegative functional,
\begin{align}
o(1)&={}_{(S^1)'}\langle J_f'(u_n),(u_n)_-\rangle_{S^1}
=\langle u_n,(u_n)_-\rangle_{S^1}-{}_{(S^1)'}\langle f,(u_n)_-\rangle_{S^1}
\leq-\|(u_n)_-\|_{S^1}^2.\nonumber
\end{align}
Hence $(u_n)_-\to0$ in $S^1(\hn)$ and, up to a subsequence, a.e. in $\hn$, so that $(u_0)_-=0$ a.e. in $\hn$. Therefore, without loss of generality, $\{u_n\}$ may be taken to be a nonnegative $(PS)_{c_0}$ sequence for $J_f$. This completes Step~2.

\noindent {\bf Step~3:}  
In this step, we prove that $u_n\to u_0$ in $S^1(\hn)$ and $u_0\in\Upsilon_1$. By Theorem~\ref{thm_PS_decomp},
\begin{equation}\label{J5}
u_n-\bigg(u_0+\sum_{j=1}^{m}D_{\xi_n^j,\lambda_n^j}\tilde u_j\bigg)\to0
\quad\text{in }S^1(\hn)\text{ as }n\to\infty,
\end{equation}
where $J_f'(u_0)=0$, each $\tilde u_j\geq0$ is a nontrivial solution of \eqref{eq:limiting_eq_j} with the constant $a_\infty^{(j)}$, and the sequences $\{\xi_n^j\}_n\subset\hn$, $\{\lambda_n^j\}_n\subset\R^+$ satisfy \eqref{eq:param_behavior}. It remains to show that $m=0$. Suppose, to the contrary, that $m\geq1$ in \eqref{J5}. Then, for each $1\leq j\leq m$,
\begin{align}\label{16-4-3}
\Psi\left(D_{\xi_n^j,\lambda_n^j}\tilde u_j\right)
&=\|\tilde u_j\|_{S^1}^2-p\int_{\hn}a\,|D_{\xi_n^j,\lambda_n^j}\tilde u_j|^{2^{\star}}\,d\xi
\longrightarrow(1-p)\|\tilde u_j\|_{S^1}^2<0,
\end{align}
Moreover, Theorem~\ref{thm_PS_decomp}--$(5)$ gives
\[
c_0=\lim_{n\to\infty}J_f(u_n)
=J_f(u_0)+\sum_{j=1}^{m}\mathcal{E}_{a_\infty^{(j)},0}(\tilde u_j),
\]

Here, in the first line, we have used that $\int_{\hn}a|D_{\xi_n^j,\lambda_n^j}\tilde u_j|^{2^\star}d\xi\to a_\infty^{(j)}\int_{\hn}|\tilde u_j|^{2^\star}d\xi=\|\tilde u_j\|_{S^1}^2$ (see Step~V in the proof of Theorem~\ref{thm_PS_decomp}). Since $\mathcal{E}_{a_\infty^{(j)},0}(\tilde u_j)>0$ by \eqref{eq_energy_drop}, we get $J_f(u_0)<c_0$. Therefore, $u_0\notin\Upsilon_1$, since $c_0=\inf_{\Upsilon_1}J_f(u)$, and
\begin{equation}\label{16-4-4}
\Psi(u_0)\leq0.
\end{equation}
We next estimate $\Psi\left(u_0+\sum_{j=1}^{m}D_{\xi_n^j,\lambda_n^j}\tilde u_j\right)$. Since $\{u_n\}\subset \Upsilon_1$, we have $\Psi(u_n)\geq0$. Combining this with the uniform continuity of $\Psi$ and \eqref{J5}, we obtain
\begin{equation}\label{J8}
0\leq\liminf_{n\to\infty}\Psi(u_n)
=\liminf_{n\to\infty}\Psi\left(u_0+\sum_{j=1}^{m}D_{\xi_n^j,\lambda_n^j}\tilde u_j\right).
\end{equation}
By Step~2, we already know that $u_0\geq0$, while $\tilde u_j\geq0$ for every $j$. Therefore,    
\begin{align}\label{20-4-1}
\Psi\bigg(u_0 +\sum_{j=1}^{m} \big(D_{ \xi_n^j,\lambda_n^j}\tilde{u}_j\big)\bigg)\nonumber &\leq\|u_0\|_{S^1}^2+\left\|\sum_{j=1}^{m} \big(D_{\xi_n^j,\lambda_n^j}\tilde{u}_j\big)\right\|^2_{S^1}+2\left\langle u_0,\sum_{j=1}^{m} \left(D_{ \xi_n^j,\lambda_n^j}\tilde{u}_j\right) \right\rangle_{S^1}\nonumber\\
&- p\bigg(\int_{\hn}a|u_0|^{2^{\star}}\,d\xi +\int_{\hn}a\left|\sum_{j=1}^{m} \big(D_{ \xi_n^j,\lambda_n^j}\tilde{u}_j\big)\right|^{2^{\star}}\,d\xi\bigg)\nonumber\\
&=\Psi(u_0)+\Psi\bigg(\sum_{j=1}^{m} \big(D_{ \xi_n^j,\lambda_n^j}\tilde{u}_j\big)\bigg)+2\big\langle u_0,\sum_{j=1}^{m} \big(D_{ \xi_n^j,\lambda_n^j}\tilde{u}_j\big) \big\rangle_{S^1}.
\end{align}

A similar argument also shows that
\begin{align}\label{20-4-22}
\Psi\left(\sum_{j=1}^{m} \big(D_{\xi_n^j,\lambda_n^j}\tilde{u}_j\big)\right) &\leq \sum_{j=1}^{m} \Psi\left( \big(D_{\xi_n^j,\lambda_n^j}\tilde{u}_j\big)\right) +2\sum_{\substack{1\leq i<j\leq m}}^{}\bigg\langle \big(D_{\xi_n^i,\lambda_n^i}\tilde{u}_i\big) , \big(D_{\xi_n^j,\lambda_n^j}\tilde{u}_j\big) \bigg\rangle_{S^1}.
\end{align}

The estimates \eqref{20-4-1} and \eqref{20-4-22} together imply
\begin{align*}
\Psi\bigg(u_0 +\sum_{j=1}^{m} \big(D_{\xi_n^j,\lambda_n^j}\tilde{u}_j\big)\bigg)&\leq \Psi(u_0)+\sum_{j=1}^{m} \Psi\bigg( \big(D_{\xi_n^j,\lambda_n^j}\tilde{u}_j\big)\bigg)\nonumber\\
&+2\sum_{\substack{1\leq i<j\leq m}}^{m}\bigg\langle \big(D_{\xi_n^i,\lambda_n^i}\tilde{u}_i\big) , \big(D_{\xi_n^j,\lambda_n^j}\tilde{u}_j\big) \bigg\rangle_{S^1}+ 2\big\langle u_0,\sum_{j=1}^{m} \big(D_{\xi_n^j,\lambda_n^j}\tilde{u}_j\big) \big\rangle_{S^1}.
\end{align*}

Now, we claim the following:
\begin{enumerate}
    \item[(i)] $\left\langle u_0 ,\,  \left(D_{\xi_n^j,\lambda_n^j}\tilde{u}_j \right)\right\rangle_{S^1}= o(1)$  for $j=1,\cdots, \,m$;
    \item[(ii)] $\left\langle  D_{\xi_n^i,\lambda_n^i}\tilde{u}_i,\,  \left(D_{\xi_n^j,\lambda_n^j}\tilde{u}_j\right) \right\rangle_{S^1}=o(1)$,  for $i\neq j$.
\end{enumerate}

For $(i)$, define $u_0^n(\xi)\coloneqq(\lambda_n^j)^{\frac{Q-2}{2}}u_0\left((\xi_n^j)\circ\delta_{\lambda_n^j}(\xi)\right)$. Since $\left|\ln(\lambda_n^j)\right|+\|\xi_n^j\|\to\infty$ and $u_0\in S^1(\hn)$, we have $u_0^n\rightharpoonup0$ in $S^1(\hn)$ as $n\to\infty$. So,
\begin{align*}
\left\langle u_0 ,\,  \left(D_{\xi_n^j,\lambda_n^j}\tilde{u}_j \right)\right\rangle_{S^1}
&= (\lambda_n^j)^{-\frac{Q-2}{2}}\int_{\hn}\nabla_{\hn}u_0\cdot \nabla_{\hn} \tilde{u}_j\left(\delta_{\frac{1}{\lambda_n^j}}\left((\xi_n^j)^{-1}\circ\xi\right)\right)\,d\xi\\
&=  (\lambda_n^j)^\frac{Q-2}{2}\int_{\hn}\nabla_{\hn}u_0\left((\xi_n^j)\circ\delta_{\lambda_n^j}(\xi)\right)\cdot\nabla_{\hn} \tilde{u}_j(\xi)\,d\xi\\
&=  \left\langle u_0^n ,\tilde{u}_j \right\rangle_{S^1} = o(1).
\end{align*}
In the same manner, for $i\neq j$, we can show
\begin{align*}
\left\langle D_{\xi_n^i,\lambda_n^i}\tilde{u}_i,D_{\xi_n^j,\lambda_n^j}\tilde u_j\right\rangle_{S^1}
=
\left\langle D_{\xi_n^j,\lambda_n^j}\tilde u_j,D_{\xi_n^i,\lambda_n^i}\tilde{u}_i\right\rangle_{S^1}
=o(1).
\end{align*}
This follows from Lemma~\ref{pp-lemma-3}, Theorem~\ref{thm_PS_decomp}-$(3)$, and the claim~$(ii)$ follows
\[
\left|\log\left(\frac{\lambda_n^{i}}{\lambda_n^{j}}\right)\right|
+\left\|\delta_{\frac{1}{\lambda_n^{j}}}\!\left((\xi_n^{j})^{-1}\circ\xi_n^{i}\right)\right\|
\to\infty.
\]

Combining the above claim with \eqref{16-4-3} and \eqref{16-4-4} contradicts \eqref{J8}. Hence, $m=0$ in \eqref{J5}, and therefore $u_n\to u_0$ in $S^1(\hn)$. Consequently, $\Psi(u_n)\to\Psi(u_0)$, which implies $u_0\in\overline{\Upsilon}_1$. Since $c_0<c_1$, we further obtain $u_0\in\Upsilon_1$. This completes Step~3.

By the preceding steps, $J_f(u_0)=c_0$ and $J_f'(u_0)=0$. Hence, $u_0$ is a weak solution of \eqref{pos_eqn_pert_Hn}. Combining this with Remark \ref{rmk_equi_cric_pt}, we complete the proof.
\end{proof}

\begin{proposition}\label{sec-cric-pt}
Assume that \eqref{J1:3} and \hyperref[assum_max]{\textbf{(M)}} hold. Then $J_{f}$ possesses a second critical point $u_1\neq u_0$, where $u_0\in \Upsilon_1$ is the critical point obtained in Proposition~\ref{fir-cri-pt}. Particularly, $u_1$ is a second positive solution of \eqref{eqn_pert_Hn}.
\end{proposition}
\begin{proof}
Let $\zeta$, $C$, $r$ and $\beta$ be as in \hyperref[assum_max]{\textbf{(M)}}. Decreasing $r$ and $\beta$ if necessary, we may assume that $r\leq1$ and $\frac{Q-2}{2}<\beta<Q$. We write $a_M:=\|a\|_{L^\infty}=a(\zeta)$, and for $\varepsilon\in(0,r)$, we define
\[
    w_\varepsilon:=a_M^{-\frac{Q-2}{4}}D_{\zeta,\varepsilon}U,\qquad A:=\|w_\varepsilon\|_{S^1}^2=a_M^{-\frac{Q-2}{2}}\mathcal{S}_Q^{Q/2},\qquad B_\varepsilon:=\int_{\mathbb{H}^n}a\,w_\varepsilon^{2^\star}\,d\xi.
\]
Since $d(\zeta,\zeta\circ\delta_\varepsilon(\eta))=\varepsilon\|\eta\|$ and $U(\eta)^{2^\star}\leq C(1+\|\eta\|)^{-2Q}$, the change of variables $\xi=\zeta\circ\delta_\varepsilon(\eta)$ gives
\begin{equation}\label{eq_Beps}
    0\leq A-B_\varepsilon=a_M^{-\frac{Q}{2}}\int_{\mathbb{H}^n}\big(a_M-a(\zeta\circ\delta_\varepsilon(\eta))\big)U(\eta)^{2^\star}\,d\eta
    \leq C\varepsilon^{\beta}\int_{\mathbb{H}^n}\|\eta\|^{\beta}U^{2^\star}d\eta+C\int_{\|\eta\|\geq r/\varepsilon}U^{2^\star}d\eta\leq C\varepsilon^{\beta}.
\end{equation}
We now proceed in four steps.

\emph{Step 1.} Since $J_f'(u_0)=0$ and $u_0,w_\varepsilon>0$, we have $\langle u_0,w_\varepsilon\rangle_{S^1}=\int_{\mathbb{H}^n}a\,u_0^{p}w_\varepsilon\,d\xi+{}_{(S^1)'}\langle f,w_\varepsilon\rangle_{S^1}$, and therefore, for all $t\geq0$,
\begin{equation}\label{eq_ray}
    J_f(u_0+tw_\varepsilon)=J_f(u_0)+g_\varepsilon(t)-\frac{1}{2^\star}\int_{\mathbb{H}^n}a\,G(u_0,tw_\varepsilon)\,d\xi,\qquad g_\varepsilon(t):=\frac{t^2}{2}A-\frac{t^{2^\star}}{2^\star}B_\varepsilon,
\end{equation}
where $G(s,\tau):=(s+\tau)^{2^\star}-s^{2^\star}-\tau^{2^\star}-2^\star s^{p}\tau$. For $s,\tau>0$, we can write $G(s,\tau)=\tau^{2^\star}h(s/\tau)$ with $h(x):=(1+x)^{2^\star}-1-x^{2^\star}-2^\star x^{p}$. Since $h(0)=0$ and, by the mean value theorem, $h'(x)=2^\star\big[(1+x)^{p}-x^{p}-px^{p-1}\big]>0$ for $x>0$, we get $G\geq0$. Moreover, $h(x)/x\to2^\star$ as $x\to0$, and hence there exists $c>0$ such that $h(x)\geq cx$ on $(0,1]$, that is,
\begin{equation}\label{eq_G_lower}
    G(s,\tau)\geq c\,s\,\tau^{p}\qquad\text{for }0<s\leq\tau.
\end{equation}
By \eqref{eq_Beps}, we have $B_\varepsilon\to A$, and therefore there exist $0<t_1<t_2$, independent of small $\varepsilon$, such that $g_\varepsilon(t)\leq\frac{A}{2Q}$ for $t\leq t_1$ and $g_\varepsilon(t)<0$ for $t\geq t_2$. We also have $\max_{t\geq0}g_\varepsilon=\frac1Q A^{\frac Q2}B_\varepsilon^{-\frac{Q-2}{2}}\leq\frac1QA+C\varepsilon^{\beta}$.

\emph{Step 2.} We claim that $\int_{\mathbb{H}^n}a\,G(u_0,tw_\varepsilon)\,d\xi\geq c\,\varepsilon^{\frac{Q-2}{2}}$ for all $t\in[t_1,t_2]$ and small $\varepsilon$. Since $u_0\in S^1(\mathbb{H}^n)$ is positive and $\mathcal{L}_{\mathbb{H}^n}u_0\geq0$, the weak Harnack inequality \cite[Proposition~2.2]{BGV26} gives $m_0:=\operatorname{ess\,inf}_{B(\zeta,r)}u_0>0$. On the other hand, on $B(\zeta,\varepsilon)$ we have $w_\varepsilon\geq c_1\varepsilon^{-\frac{Q-2}{2}}$ with $c_1:=a_M^{-\frac{Q-2}{4}}\min_{B(e,1)}U$. Let us define $E_\varepsilon:=\{\xi\in B(\zeta,\varepsilon):u_0(\xi)\leq t_1c_1\varepsilon^{-\frac{Q-2}{2}}\}$. Then, by Chebyshev's inequality, $|B(\zeta,\varepsilon)\setminus E_\varepsilon|\leq(t_1c_1)^{-2^\star}\varepsilon^{Q}\int_{B(\zeta,\varepsilon)}u_0^{2^\star}d\xi=o(\varepsilon^Q)$, and hence $|E_\varepsilon|\geq\frac12|B(e,1)|\varepsilon^Q$ for small $\varepsilon$. Now, on $E_\varepsilon$ and for $t\in[t_1,t_2]$, \eqref{eq_G_lower} gives $G(u_0,tw_\varepsilon)\geq c\,m_0(t_1c_1)^{p}\varepsilon^{-\frac{Q+2}{2}}$, and integrating over $E_\varepsilon$, we obtain our claim, since $\inf a>0$.

\emph{Step 3.} Using \eqref{eq_ray}, $G\geq0$ and Steps~1--2, we see that for small $\varepsilon$, $J_f(u_0+tw_\varepsilon)\leq J_f(u_0)+\frac{A}{2Q}$ for $t\leq t_1$, $J_f(u_0+tw_\varepsilon)\leq J_f(u_0)+\frac1QA+C\varepsilon^{\beta}-c\,\varepsilon^{\frac{Q-2}{2}}$ for $t\in[t_1,t_2]$, and $J_f(u_0+tw_\varepsilon)<J_f(u_0)$ for $t\geq t_2$. Since $\beta>\frac{Q-2}{2}$, we can fix $\varepsilon$ such that $\sup_{t\geq0}J_f(u_0+tw_\varepsilon)<J_f(u_0)+\frac1QA$. Moreover, as $u_0,w_\varepsilon\geq0$, Young's inequality gives $\Psi(u_0+Tw_\varepsilon)\leq2\|u_0\|_{S^1}^2+2T^2A-pT^{2^\star}B_\varepsilon<0$ for $T\geq t_2$ large, and we fix such a $T$. Let us define
\[
    \gamma:=\inf_{\kappa\in\Gamma}\max_{s\in[0,1]}J_f(\kappa(s)),\qquad \Gamma:=\left\{\kappa\in C([0,1],S^1(\mathbb{H}^n)):\kappa(0)=u_0,\ \kappa(1)=u_0+Tw_\varepsilon\right\}.
\]
Since $u_0\in\Upsilon_1$ and $u_0+Tw_\varepsilon\in\Upsilon_2$, every $\kappa\in\Gamma$ meets $\Upsilon$, and hence $\gamma\geq c_1>c_0=J_f(u_0)$ by Lemma~\ref{lem_strict_less}. Using the path $s\mapsto u_0+sTw_\varepsilon$, we thus obtain
\begin{equation}\label{eq_gamma_bounds}
    J_f(u_0)<\gamma<J_f(u_0)+\frac{1}{Q}\,a_M^{-\frac{Q-2}{2}}\mathcal{S}_Q^{Q/2}.
\end{equation}

\emph{Step 4.} By the mountain-pass theorem in the form of Ekeland's variational principle, there exists a $(PS)$ sequence $\{u_n\}$ for $J_f$ at level $\gamma$. This sequence is bounded in $S^1(\mathbb{H}^n)$, since $J_f(u_n)-\frac{1}{2^\star}{}_{(S^1)'}\langle J_f'(u_n),u_n\rangle_{S^1}=\frac1Q\|u_n\|_{S^1}^2-\big(1-\frac{1}{2^\star}\big){}_{(S^1)'}\langle f,u_n\rangle_{S^1}$. Arguing exactly as in Step~2 of the proof of Proposition~\ref{fir-cri-pt}, we get $(u_n)_-\to0$ in $S^1(\mathbb{H}^n)$. Hence $J_f(u_n)-\mathcal{E}_{a,f}(u_n)\to0$ and $J_f'(u_n)-\mathcal{E}_{a,f}'(u_n)=-a\,(u_n)_-^{p}\to0$ in $(S^1(\mathbb{H}^n))'$, which means that $\{u_n\}$ is also a $(PS)_\gamma$ sequence for $\mathcal{E}_{a,f}$. Therefore, by Theorem~\ref{thm_PS_decomp}, a subsequence converges weakly to a critical point $v_0$ of $\mathcal{E}_{a,f}$, with $v_0\geq0$ since $(u_n)_-\to0$, which is also a critical point of $J_f$, and
\[
    \gamma=J_f(v_0)+\sum_{j=1}^m\mathcal{E}_{a_\infty^{(j)},0}(u^{(j)}).
\]
If $v_0\neq u_0$, we set $u_1:=v_0$, and we are done. Let us now assume that $v_0=u_0$. If $m=0$, then $\gamma=J_f(u_0)$, which is not possible by \eqref{eq_gamma_bounds}. On the other hand, if $m\geq1$, then \eqref{eq_energy_drop} together with $a_\infty^{(j)}\leq a_M$ gives $\gamma\geq J_f(u_0)+\frac1Q a_M^{-\frac{Q-2}{2}}\mathcal{S}_Q^{Q/2}$, which again contradicts \eqref{eq_gamma_bounds}. Finally, $u_1$ is positive by Remark~\ref{rmk_equi_cric_pt}. This completes the proof.
\end{proof}

We note that Propositions~\ref{fir-cri-pt} and \ref{sec-cric-pt} were proved under the assumption \eqref{J1:3}. In the following lemma, we show that this condition is satisfied whenever $f$ has sufficiently small norm.

\begin{lemma}\label{lem_J1:3}
If \eqref{eq:f_small} holds, then \eqref{J1:3} holds.
\end{lemma}
\begin{proof}
Let $u\in S^1(\mathbb{H}^n)$ be such that $\int_{\mathbb{H}^n}a|u|^{2^\star}d\xi=1$. Then $\|u\|_{L^{2^\star}}^{2^\star}\geq\|a\|_{L^\infty}^{-1}$, and hence the Folland--Stein inequality gives $\|u\|_{S^1}^{Q/2}\geq\mathcal{S}_Q^{Q/4}\|u\|_{L^{2^\star}}^{Q/2}\geq\mathcal{S}_Q^{Q/4}\|a\|_{L^\infty}^{-\frac{Q-2}{4}}$. Consequently,
\[
    C_0\|u\|_{S^1}^{\frac{Q+2}{2}}-{}_{(S^1)'}\langle f,u\rangle_{S^1}\geq\|u\|_{S^1}\Big(C_0\,\|a\|_{L^\infty}^{-\frac{Q-2}{4}}\mathcal{S}_Q^{Q/4}-\|f\|_{(S^1)'}\Big),
\]
and the right-hand side is bounded below by a positive constant, since $\|u\|_{S^1}\geq\mathcal{S}_Q^{1/2}\|a\|_{L^\infty}^{-1/2^\star}$.
\end{proof}

\begin{proof}[\textbf{Proof of Theorem \ref{thm_multi_pos_sol}}]
Combining Proposition~\ref{fir-cri-pt} and Proposition~\ref{sec-cric-pt} with Lemma~\ref{lem_J1:3}, we conclude the proof of Theorem \ref{thm_multi_pos_sol}.
\end{proof}

\smallskip

\noindent {\bf Acknowledgments.} 
R.~Basak is supported by the National Science and Technology Council of Taiwan under research grant number 113-2811-M-003-014. S.~Chakraborty thanks GITAM, Hyderabad, for its warm hospitality and an excellent research environment. T.~Rana is supported by the FWO Odysseus 1 grant G.0H94.18N: Analysis and Partial Differential Equations, the Methusalem program of the Ghent University Special Research Fund (BOF), (T.~Rana Project title: BOFMET2021000601). T.~Rana is also supported by a BOF postdoctoral fellowship at Ghent University BOF24/PDO/025. P.~Roychowdhury is supported by the Seed Grant provided by the Indian Institute of Technology Hyderabad under Grant No.~SG/IITH/F391/2026-27/SG-230.

\smallskip

\noindent
{\bf Data availability.} 
Data sharing does not apply to this article as no datasets were
generated or analysed during the current study.

\smallskip

\noindent
{\bf Conflict of interest.} On behalf of all authors, the corresponding author states that there is no conflict of interest.

\bibliographystyle{alphaurl}
\bibliography{Arxiv_v2}

@book {VSCC92,
    AUTHOR = {Varopoulos, N. Th. and Saloff-Coste, L. and Coulhon, T.},
     TITLE = {Analysis and geometry on groups},
    SERIES = {Cambridge Tracts in Mathematics},
    VOLUME = {100},
 PUBLISHER = {Cambridge University Press, Cambridge},
      YEAR = {1992},
     PAGES = {xii+156},
      ISBN = {0-521-35382-3},
   MRCLASS = {43A80 (47D03 47F05 58G11 60B15)},
  MRNUMBER = {1218884},
MRREVIEWER = {A.\ Hulanicki},
}

@article {GV00,
    AUTHOR = {Garofalo, Nicola and Vassilev, Dimiter},
     TITLE = {Regularity near the characteristic set in the non-linear
              {D}irichlet problem and conformal geometry of sub-{L}aplacians
              on {C}arnot groups},
   JOURNAL = {Math. Ann.},
  FJOURNAL = {Mathematische Annalen},
    VOLUME = {318},
      YEAR = {2000},
    NUMBER = {3},
     PAGES = {453--516},
      ISSN = {0025-5831,1432-1807},
   MRCLASS = {35H30 (35D05 58J05)},
  MRNUMBER = {1800766},
MRREVIEWER = {Luca\ Capogna},
       DOI = {10.1007/s002080000127},
       URL = {https://doi.org/10.1007/s002080000127},
}

@misc{Liu25,
      title={{CR} {Yamabe} Equation on the {Heisenberg} Group via the method of moving spheres}, 
      author={Congwen Liu},
      year={2025},
      eprint={2512.22458},
      archivePrefix={arXiv},
      primaryClass={math.AP},
      url={https://arxiv.org/abs/2512.22458}, 
}

@article {GL92,
    AUTHOR = {Garofalo, Nicola and Lanconelli, Ermanno},
     TITLE = {Existence and nonexistence results for semilinear equations on
              the {H}eisenberg group},
   JOURNAL = {Indiana Univ. Math. J.},
  FJOURNAL = {Indiana University Mathematics Journal},
    VOLUME = {41},
      YEAR = {1992},
    NUMBER = {1},
     PAGES = {71--98},
      ISSN = {0022-2518,1943-5258},
   MRCLASS = {35J65 (22E25)},
  MRNUMBER = {1160903},
MRREVIEWER = {Jana\ D.\ Madjarova},
       DOI = {10.1512/iumj.1992.41.41005},
       URL = {https://doi.org/10.1512/iumj.1992.41.41005},
}

@article {PP15,
    AUTHOR = {Palatucci, G. and Pisante, A.},
     TITLE = {A global compactness type result for {P}alais-{S}male
              sequences in fractional {S}obolev spaces},
   JOURNAL = {Nonlinear Anal.},
  FJOURNAL = {Nonlinear Analysis. Theory, Methods \& Applications. An
              International Multidisciplinary Journal},
    VOLUME = {117},
      YEAR = {2015},
     PAGES = {1--7},
      ISSN = {0362-546X,1873-5215},
   MRCLASS = {35A15 (35B33 35C20 35J20 35J61 49J45)},
  MRNUMBER = {3316602},
MRREVIEWER = {Changpin\ Li},
       DOI = {10.1016/j.na.2014.12.027},
       URL = {https://doi.org/10.1016/j.na.2014.12.027},
}

@article {Jaf99,
    AUTHOR = {Jaffard, S.},
     TITLE = {Analysis of the lack of compactness in the critical {S}obolev
              embeddings},
   JOURNAL = {J. Funct. Anal.},
  FJOURNAL = {Journal of Functional Analysis},
    VOLUME = {161},
      YEAR = {1999},
    NUMBER = {2},
     PAGES = {384--396},
      ISSN = {0022-1236,1096-0783},
   MRCLASS = {46E35},
  MRNUMBER = {1674639},
MRREVIEWER = {Stanis\l aw\ W\polhk edrychowicz},
       DOI = {10.1006/jfan.1998.3364},
       URL = {https://doi.org/10.1006/jfan.1998.3364},
}

@article {Ger98,
    AUTHOR = {G{\'e}rard, P.},
     TITLE = {Description du d\'efaut de compacit\'e{} de l'injection de
              {S}obolev},
   JOURNAL = {ESAIM Control Optim. Calc. Var.},
  FJOURNAL = {ESAIM. Control, Optimisation and Calculus of Variations.
              European Series in Applied and Industrial Mathematics},
    VOLUME = {3},
      YEAR = {1998},
     PAGES = {213--233},
      ISSN = {1292-8119,1262-3377},
   MRCLASS = {46E35 (49K10)},
  MRNUMBER = {1632171},
MRREVIEWER = {W.\ P.\ Ziemer},
       DOI = {10.1051/cocv:1998107},
       URL = {https://doi.org/10.1051/cocv:1998107},
}

@article {Sol95,
    AUTHOR = {Solimini, S.},
     TITLE = {A note on compactness-type properties with respect to
              {L}orentz norms of bounded subsets of a {S}obolev space},
   JOURNAL = {Ann. Inst. H. Poincar\'e{} C Anal. Non Lin\'eaire},
  FJOURNAL = {Annales de l'Institut Henri Poincar\'e{} C. Analyse Non
              Lin\'eaire},
    VOLUME = {12},
      YEAR = {1995},
    NUMBER = {3},
     PAGES = {319--337},
      ISSN = {0294-1449,1873-1430},
   MRCLASS = {46E35},
  MRNUMBER = {1340267},
MRREVIEWER = {Nenad\ Antoni\'c},
       DOI = {10.1016/S0294-1449(16)30159-7},
       URL = {https://doi.org/10.1016/S0294-1449(16)30159-7},
}

@book {Thanbook,
    AUTHOR = {Thangavelu, Sundaram},
     TITLE = {An introduction to the uncertainty principle},
    SERIES = {Progress in Mathematics},
    VOLUME = {217},
      NOTE = {Hardy's theorem on Lie groups,
              With a foreword by Gerald B.\ Folland},
 PUBLISHER = {Birkh\"auser Boston, Inc., Boston, MA},
      YEAR = {2004},
     PAGES = {xiv+174},
      ISBN = {0-8176-4330-3},
   MRCLASS = {43A80 (33C45 33C55 42C10 43-02 43A20)},
  MRNUMBER = {2008480},
MRREVIEWER = {Weixing\ Zheng},
       DOI = {10.1007/978-0-8176-8164-7},
       URL = {https://doi.org/10.1007/978-0-8176-8164-7},
}

@article {C11,
    AUTHOR = {Chamorro, Diego},
     TITLE = {Improved {S}obolev inequalities and {M}uckenhoupt weights on
              stratified {L}ie groups},
   JOURNAL = {J. Math. Anal. Appl.},
  FJOURNAL = {Journal of Mathematical Analysis and Applications},
    VOLUME = {377},
      YEAR = {2011},
    NUMBER = {2},
     PAGES = {695--709},
      ISSN = {0022-247X,1096-0813},
   MRCLASS = {43A80 (42B35 46E35)},
  MRNUMBER = {2769168},
MRREVIEWER = {Rudra\ P.\ Sarkar},
       DOI = {10.1016/j.jmaa.2010.11.047},
       URL = {https://doi.org/10.1016/j.jmaa.2010.11.047},
}

@article{BN83,
  author  = {Br{\'e}zis, Ha{\"\i}m and Nirenberg, Louis},
  title   = {Positive solutions of nonlinear elliptic equations involving critical {Sobolev} exponents},
  journal = {Communications on Pure and Applied Mathematics},
  volume  = {36},
  number  = {4},
  pages   = {437--477},
  year    = {1983},
  doi     = {10.1002/cpa.3160360405}
}

@article {Lio84b,
    AUTHOR = {Lions, P.-L.},
     TITLE = {The concentration-compactness principle in the calculus of
              variations. {T}he locally compact case. {II}},
   JOURNAL = {Ann. Inst. H. Poincar\'e{} Anal. Non Lin\'eaire},
  FJOURNAL = {Annales de l'Institut Henri Poincar\'e. Analyse Non
              Lin\'eaire},
    VOLUME = {1},
      YEAR = {1984},
    NUMBER = {4},
     PAGES = {223--283},
      ISSN = {0294-1449},
   MRCLASS = {49A50 (49A22)},
  MRNUMBER = {778974},
MRREVIEWER = {Gianfranco\ Bottaro},
       URL = {http://www.numdam.org/item?id=AIHPC_1984__1_4_223_0},
}

@article{Str84,
  author  = {Struwe, Michael},
  title   = {A global compactness result for elliptic boundary value problems involving limiting nonlinearities},
  journal = {Mathematische Zeitschrift},
  volume  = {187},
  number  = {4},
  pages   = {511--517},
  year    = {1984},
  doi     = {10.1007/BF01174186}
}

@article{BC88,
  author  = {Bahri, Abbas and Coron, Jean-Michel},
  title   = {On a nonlinear elliptic equation involving the critical {Sobolev} exponent: the effect of the topology of the domain},
  journal = {Communications on Pure and Applied Mathematics},
  volume  = {41},
  number  = {3},
  pages   = {253--294},
  year    = {1988},
  doi     = {10.1002/cpa.3160410302}
}

@article {Tar92,
    AUTHOR = {Tarantello, G.},
     TITLE = {On nonhomogeneous elliptic equations involving critical
              {S}obolev exponent},
   JOURNAL = {Ann. Inst. H. Poincar\'e{} C Anal. Non Lin\'eaire},
  FJOURNAL = {Annales de l'Institut Henri Poincar\'e{} C. Analyse Non
              Lin\'eaire},
    VOLUME = {9},
      YEAR = {1992},
    NUMBER = {3},
     PAGES = {281--304},
      ISSN = {0294-1449,1873-1430},
   MRCLASS = {35J20 (35J65)},
  MRNUMBER = {1168304},
MRREVIEWER = {Heinrich\ Begehr},
       DOI = {10.1016/S0294-1449(16)30238-4},
       URL = {https://doi.org/10.1016/S0294-1449(16)30238-4},
}

@article {LP87,
    AUTHOR = {Lee, John M. and Parker, Thomas H.},
     TITLE = {The {Y}amabe problem},
   JOURNAL = {Bull. Amer. Math. Soc. (N.S.)},
  FJOURNAL = {American Mathematical Society. Bulletin. New Series},
    VOLUME = {17},
      YEAR = {1987},
    NUMBER = {1},
     PAGES = {37--91},
      ISSN = {0273-0979,1088-9485},
   MRCLASS = {53-02 (35J60 53A30 53C20 58E15 58G30)},
  MRNUMBER = {888880},
MRREVIEWER = {J.\ L.\ Kazdan},
       DOI = {10.1090/S0273-0979-1987-15514-5},
       URL = {https://doi.org/10.1090/S0273-0979-1987-15514-5},
}

@article {Lio84a,
    AUTHOR = {Lions, P.-L.},
     TITLE = {The concentration-compactness principle in the calculus of
              variations. {T}he locally compact case. {I}},
   JOURNAL = {Ann. Inst. H. Poincar\'e{} Anal. Non Lin\'eaire},
  FJOURNAL = {Annales de l'Institut Henri Poincar\'e. Analyse Non
              Lin\'eaire},
    VOLUME = {1},
      YEAR = {1984},
    NUMBER = {2},
     PAGES = {109--145},
      ISSN = {0294-1449},
   MRCLASS = {49A50 (49A22)},
  MRNUMBER = {778970},
MRREVIEWER = {Gianfranco\ Bottaro},
       URL = {http://www.numdam.org/item?id=AIHPC_1984__1_2_109_0},
}

@article {Sch84,
    AUTHOR = {Schoen, Richard},
     TITLE = {Conformal deformation of a {R}iemannian metric to constant
              scalar curvature},
   JOURNAL = {J. Differential Geom.},
  FJOURNAL = {Journal of Differential Geometry},
    VOLUME = {20},
      YEAR = {1984},
    NUMBER = {2},
     PAGES = {479--495},
      ISSN = {0022-040X,1945-743X},
   MRCLASS = {58G30 (53C20 53C21)},
  MRNUMBER = {788292},
MRREVIEWER = {Raymond\ Barre},
       URL = {http://projecteuclid.org/euclid.jdg/1214439291},
}

@article {Aub76,
    AUTHOR = {Aubin, Thierry},
     TITLE = {Probl\`emes isop\'erim\'etriques et espaces de {S}obolev},
   JOURNAL = {J. Differential Geometry},
  FJOURNAL = {Journal of Differential Geometry},
    VOLUME = {11},
      YEAR = {1976},
    NUMBER = {4},
     PAGES = {573--598},
      ISSN = {0022-040X,1945-743X},
   MRCLASS = {58D15 (46E35 53C20)},
  MRNUMBER = {448404},
MRREVIEWER = {J.\ L.\ Kazdan},
       URL = {http://projecteuclid.org/euclid.jdg/1214433725},
}

@article {Tal76,
    AUTHOR = {Talenti, Giorgio},
     TITLE = {Best constant in {S}obolev inequality},
   JOURNAL = {Ann. Mat. Pura Appl. (4)},
  FJOURNAL = {Annali di Matematica Pura ed Applicata. Serie Quarta},
    VOLUME = {110},
      YEAR = {1976},
     PAGES = {353--372},
      ISSN = {0003-4622},
   MRCLASS = {46E35},
  MRNUMBER = {463908},
MRREVIEWER = {L.\ Cattabriga},
       DOI = {10.1007/BF02418013},
       URL = {https://doi.org/10.1007/BF02418013},
}

@article {PPT25,
    AUTHOR = {Palatucci, G. and Piccinini, M. and Temperini,
              L.},
     TITLE = {Struwe's global compactness and energy approximation of the
              critical {S}obolev embedding in the {H}eisenberg group},
   JOURNAL = {Adv. Calc. Var.},
  FJOURNAL = {Advances in Calculus of Variations},
    VOLUME = {18},
      YEAR = {2025},
    NUMBER = {3},
     PAGES = {731--754},
      ISSN = {1864-8258,1864-8266},
   MRCLASS = {35R03 (35A15 35J08 46E35 49J45)},
  MRNUMBER = {4926901},
MRREVIEWER = {Gabriel\ Ara\'ujo},
       DOI = {10.1515/acv-2024-0044},
       URL = {https://doi.org/10.1515/acv-2024-0044},
}

@article {BGV26,
    AUTHOR = {Biagi, Stefano and Pinamonti, Andrea and Vecchi, Eugenio},
     TITLE = {Critical singular problems in {Carnot} groups},
   JOURNAL = {J. Geom. Anal.},
  FJOURNAL = {Journal of Geometric Analysis},
    VOLUME = {36},
      YEAR = {2026},
    NUMBER = {8},
     PAGES = {Paper No. 268},
      ISSN = {1050-6926,1559-002X},
   MRCLASS = {35R03 (35B25 35B33 35J70)},
  MRNUMBER = {MR5099060},
       DOI = {10.1007/s12220-026-02516-8},
       URL = {https://doi.org/10.1007/s12220-026-02516-8},
}

@article {CU01,
    AUTHOR = {Citti, G. and Uguzzoni, F.},
     TITLE = {Critical semilinear equations on the {H}eisenberg group: the
              effect of the topology of the domain},
   JOURNAL = {Nonlinear Anal.},
  FJOURNAL = {Nonlinear Analysis. Theory, Methods \& Applications. An
              International Multidisciplinary Journal},
    VOLUME = {46},
      YEAR = {2001},
    NUMBER = {3},
     PAGES = {399--417},
      ISSN = {0362-546X,1873-5215},
   MRCLASS = {35J60 (35B33 58E05)},
  MRNUMBER = {1851860},
MRREVIEWER = {Dimiter\ N.\ Vassilev},
       DOI = {10.1016/S0362-546X(00)00138-3},
       URL = {https://doi.org/10.1016/S0362-546X(00)00138-3},
}

@misc{FV23,
      title={Liouville-type results for the {CR} {Yamabe} equation in the {Heisenberg} group}, 
      author={Joshua Flynn and Jérôme Vétois},
      journal={Ann. Sc. Norm. Super. Pisa Cl. Sci. (to appear)},
      year={2023},
      eprint={2310.14048},
      archivePrefix={arXiv},
      primaryClass={math.AP},
      url={https://arxiv.org/abs/2310.14048}, 
}

@book {BLU07,
    AUTHOR = {Bonfiglioli, A. and Lanconelli, E. and Uguzzoni, F.},
     TITLE = {Stratified {L}ie groups and potential theory for their
              sub-{L}aplacians},
    SERIES = {Springer Monographs in Mathematics},
 PUBLISHER = {Springer, Berlin},
      YEAR = {2007},
     PAGES = {xxvi+800},
      ISBN = {978-3-540-71896-3; 3-540-71896-6},
   MRCLASS = {22E30 (31C45 35-02 35H10 43A80)},
  MRNUMBER = {2363343},
MRREVIEWER = {Maria\ Stella\ Fanciullo},
}

@book {FS82,
    AUTHOR = {Folland, G. B. and Stein, Elias M.},
     TITLE = {Hardy spaces on homogeneous groups},
    SERIES = {Mathematical Notes},
    VOLUME = {28},
 PUBLISHER = {Princeton University Press, Princeton, NJ; University of Tokyo
              Press, Tokyo},
      YEAR = {1982},
     PAGES = {xii+285},
      ISBN = {0-691-08310-X},
   MRCLASS = {43A85 (22E45 42B30)},
  MRNUMBER = {657581},
MRREVIEWER = {Daryl\ Geller},
}

@article {Gam01,
    AUTHOR = {Gamara, Najoua},
     TITLE = {The {CR} {Y}amabe conjecture---the case {$n=1$}},
   JOURNAL = {J. Eur. Math. Soc. (JEMS)},
  FJOURNAL = {Journal of the European Mathematical Society (JEMS)},
    VOLUME = {3},
      YEAR = {2001},
    NUMBER = {2},
     PAGES = {105--137},
      ISSN = {1435-9855,1435-9863},
   MRCLASS = {32V20 (35B05 35J60 53C21 58E05)},
  MRNUMBER = {1831872},
MRREVIEWER = {John\ M.\ Lee},
       DOI = {10.1007/PL00011303},
       URL = {https://doi.org/10.1007/PL00011303},
}

@article {GY01,
    AUTHOR = {Gamara, Najoua and Yacoub, Ridha},
     TITLE = {C{R} {Y}amabe conjecture---the conformally flat case},
   JOURNAL = {Pacific J. Math.},
  FJOURNAL = {Pacific Journal of Mathematics},
    VOLUME = {201},
      YEAR = {2001},
    NUMBER = {1},
     PAGES = {121--175},
      ISSN = {0030-8730,1945-5844},
   MRCLASS = {32V20 (35B05 35J60 53C21 58E05)},
  MRNUMBER = {1867895},
MRREVIEWER = {John\ M.\ Lee},
       DOI = {10.2140/pjm.2001.201.121},
       URL = {https://doi.org/10.2140/pjm.2001.201.121},
}

@article {JL88,
    AUTHOR = {Jerison, David and Lee, John M.},
     TITLE = {Extremals for the {S}obolev inequality on the {H}eisenberg
              group and the {CR} {Y}amabe problem},
   JOURNAL = {J. Amer. Math. Soc.},
  FJOURNAL = {Journal of the American Mathematical Society},
    VOLUME = {1},
      YEAR = {1988},
    NUMBER = {1},
     PAGES = {1--13},
      ISSN = {0894-0347,1088-6834},
   MRCLASS = {53C15 (32F25 53C40 58G30)},
  MRNUMBER = {924699},
MRREVIEWER = {H.\ Jacobowitz},
       DOI = {10.2307/1990964},
       URL = {https://doi.org/10.2307/1990964},
}

@article {FS74,
    AUTHOR = {Folland, G. B. and Stein, E. M.},
     TITLE = {Estimates for the {$\bar \partial \sb{b}$} complex and
              analysis on the {H}eisenberg group},
   JOURNAL = {Comm. Pure Appl. Math.},
  FJOURNAL = {Communications on Pure and Applied Mathematics},
    VOLUME = {27},
      YEAR = {1974},
     PAGES = {429--522},
      ISSN = {0010-3640,1097-0312},
   MRCLASS = {35N15 (22E30 32K15 47G05)},
  MRNUMBER = {367477},
MRREVIEWER = {S.\ G.\ Gindikin},
       DOI = {10.1002/cpa.3160270403},
       URL = {https://doi.org/10.1002/cpa.3160270403},
}

@article {PP14,
    AUTHOR = {Palatucci, G. and Pisante, A.},
     TITLE = {Improved {S}obolev embeddings, profile decomposition, and
              concentration-compactness for fractional {S}obolev spaces},
   JOURNAL = {Calc. Var. Partial Differential Equations},
  FJOURNAL = {Calculus of Variations and Partial Differential Equations},
    VOLUME = {50},
      YEAR = {2014},
    NUMBER = {3-4},
     PAGES = {799--829},
      ISSN = {0944-2669,1432-0835},
   MRCLASS = {46E35 (35J20 35J61 35R11)},
  MRNUMBER = {3216834},
MRREVIEWER = {Giuseppe\ Di Fazio},
       DOI = {10.1007/s00526-013-0656-y},
       URL = {https://doi.org/10.1007/s00526-013-0656-y},
}

@article {Bo69,
    AUTHOR = {Bony, Jean-Michel},
     TITLE = {Principe du maximum, in\'egalite de {H}arnack et unicit\'e{}
              du probl\`eme de {C}auchy pour les op\'erateurs elliptiques
              d\'eg\'en\'er\'es},
   JOURNAL = {Ann. Inst. Fourier (Grenoble)},
  FJOURNAL = {Universit\'e{} de Grenoble. Annales de l'Institut Fourier},
    VOLUME = {19},
      YEAR = {1969},
     PAGES = {277--304 xii},
      ISSN = {0373-0956,1777-5310},
   MRCLASS = {47.65 (35.00)},
  MRNUMBER = {262881},
MRREVIEWER = {R.\ S.\ Freeman},
       URL = {http://www.numdam.org/item?id=AIF_1969__19_1_277_0},
}

@article {BP20,
    AUTHOR = {Bhakta, Mousomi and Pucci, Patrizia},
     TITLE = {On multiplicity of positive solutions for nonlocal equations
              with critical nonlinearity},
   JOURNAL = {Nonlinear Anal.},
  FJOURNAL = {Nonlinear Analysis. Theory, Methods \& Applications. An
              International Multidisciplinary Journal},
    VOLUME = {197},
      YEAR = {2020},
     PAGES = {111853, 22},
      ISSN = {0362-546X,1873-5215},
   MRCLASS = {35R11 (35A15 35B33 35J60)},
  MRNUMBER = {4079056},
MRREVIEWER = {Vincenzo\ Ambrosio},
       DOI = {10.1016/j.na.2020.111853},
       URL = {https://doi.org/10.1016/j.na.2020.111853},
}

@article {PV25,
    AUTHOR = {Prajapat, Jyotshana V. and Varghese, Anoop Skaria},
     TITLE = {Symmetry and classification of solutions to an integral
              equation in the {H}eisenberg group $\mathbb{H}^n$},
   JOURNAL = {Math. Ann.},
  FJOURNAL = {Mathematische Annalen},
    VOLUME = {392},
      YEAR = {2025},
    NUMBER = {1},
     PAGES = {659--700},
      ISSN = {0025-5831,1432-1807},
   MRCLASS = {45M20 (35B09 35J10 35R03)},
  MRNUMBER = {4887770},
MRREVIEWER = {Ram\ Krishan\ Sharma},
       DOI = {10.1007/s00208-025-03100-1},
       URL = {https://doi.org/10.1007/s00208-025-03100-1},
}

@article{CLMR25,
    AUTHOR = {Catino, Giovanni and Li, Yanyan and Monticelli, Dario D. and
              Roncoroni, Alberto},
     TITLE = {A {Liouville} theorem in the {Heisenberg} group},
   JOURNAL = {J. Eur. Math. Soc. (JEMS)},
  FJOURNAL = {Journal of the European Mathematical Society (JEMS)},
    VOLUME = {},
      YEAR = {2025},
    NUMBER = {},
     PAGES = {},
      ISSN = {},
   MRCLASS = {35J61 (32V20 35B33},
  MRNUMBER = {},
MRREVIEWER = {},
       DOI = {10.4171/JEMS/1705},
       URL = {https://doi.org/10.4171/JEMS/1705},
}

@article {BCG23,
    AUTHOR = {Bhakta, Mousomi and Chakraborty, Souptik and Ganguly, Debdip},
     TITLE = {Existence and multiplicity of positive solutions of certain
              nonlocal scalar field equations},
   JOURNAL = {Math. Nachr.},
  FJOURNAL = {Mathematische Nachrichten},
    VOLUME = {296},
      YEAR = {2023},
    NUMBER = {9},
     PAGES = {3816--3855},
      ISSN = {0025-584X,1522-2616},
   MRCLASS = {35R11 (35A15 35J60)},
  MRNUMBER = {4644540},
       DOI = {10.1002/mana.202000473},
       URL = {https://doi.org/10.1002/mana.202000473},
}

@book {AM07,
    AUTHOR = {Ambrosetti, Antonio and Malchiodi, Andrea},
     TITLE = {Nonlinear analysis and semilinear elliptic problems},
    SERIES = {Cambridge Studies in Advanced Mathematics},
    VOLUME = {104},
 PUBLISHER = {Cambridge University Press, Cambridge},
      YEAR = {2007},
     PAGES = {xii+316},
      ISBN = {978-0-521-86320-9; 0-521-86320-1},
   MRCLASS = {35J60 (35J20 47J15 47J30 58E05)},
  MRNUMBER = {2292344},
MRREVIEWER = {Sergey\ G.\ Pyatkov},
       DOI = {10.1017/CBO9780511618260},
       URL = {https://doi.org/10.1017/CBO9780511618260},
}

@incollection {Loi20,
    AUTHOR = {Loiudice, Annunziata},
     TITLE = {A multiplicity result for a non-homogeneous subelliptic
              problem with {S}obolev exponent},
 BOOKTITLE = {Advances in harmonic analysis and partial differential
              equations},
    SERIES = {Trends Math.},
     PAGES = {99--120},
 PUBLISHER = {Birkh\"auser/Springer, Cham},
      YEAR = {[2020] \copyright 2020},
      ISBN = {978-3-030-58215-9; 978-3-030-58214-2},
   MRCLASS = {35H20},
  MRNUMBER = {4174744},
       DOI = {10.1007/978-3-030-58215-9\_4},
       URL = {https://doi.org/10.1007/978-3-030-58215-9_4},
}

@article {Vas06,
    AUTHOR = {Vassilev, D.},
     TITLE = {Existence of solutions and regularity near the characteristic
              boundary for sub-{L}aplacian equations on {C}arnot groups},
   JOURNAL = {Pacific J. Math.},
  FJOURNAL = {Pacific Journal of Mathematics},
    VOLUME = {227},
      YEAR = {2006},
    NUMBER = {2},
     PAGES = {361--397},
      ISSN = {0030-8730,1945-5844},
   MRCLASS = {35H20 (35D05 35J70)},
  MRNUMBER = {2263021},
MRREVIEWER = {Maria\ Stella\ Fanciullo},
       DOI = {10.2140/pjm.2006.227.361},
       URL = {https://doi.org/10.2140/pjm.2006.227.361},
}

@article {FL12,
    AUTHOR = {Frank, R. L. and Lieb, E. H.},
     TITLE = {Sharp constants in several inequalities on the {H}eisenberg
              group},
   JOURNAL = {Ann. of Math. (2)},
  FJOURNAL = {Annals of Mathematics. Second Series},
    VOLUME = {176},
      YEAR = {2012},
    NUMBER = {1},
     PAGES = {349--381},
      ISSN = {0003-486X,1939-8980},
   MRCLASS = {43A80 (26D15 46E35 47G10)},
  MRNUMBER = {2925386},
MRREVIEWER = {Gerald\ B.\ Folland},
       DOI = {10.4007/annals.2012.176.1.6},
       URL = {https://doi.org/10.4007/annals.2012.176.1.6},
}

@article {Ben08,
    AUTHOR = {Benameur, J.},
     TITLE = {Description du d\'efaut de compacit\'e{} de l'injection de
              {S}obolev sur le groupe de {H}eisenberg},
   JOURNAL = {Bull. Belg. Math. Soc. Simon Stevin},
  FJOURNAL = {Bulletin of the Belgian Mathematical Society. Simon Stevin},
    VOLUME = {15},
      YEAR = {2008},
    NUMBER = {4},
     PAGES = {599--624},
      ISSN = {1370-1444,2034-1970},
   MRCLASS = {46E35 (35H10 43A80)},
  MRNUMBER = {2475486},
MRREVIEWER = {Luigi\ Rodino},
       URL = {http://projecteuclid.org/euclid.bbms/1225893942},
}
\end{document}